\documentclass[default]{sn-jnl}

\pdfoutput=1

\usepackage{graphicx}%
\usepackage{multirow}%
\usepackage{amsmath,amssymb,amsfonts}%
\usepackage{amsthm}%
\usepackage{mathrsfs}%
\usepackage[title]{appendix}%
\usepackage{xcolor}%
\usepackage{textcomp}%
\usepackage{manyfoot}%
\usepackage{booktabs}%
\usepackage{algorithm}%
\usepackage{algorithmicx}%
\usepackage{algpseudocode}%
\usepackage{listings}%
\newcommand{\R}{{\mathbb R}}

\newcommand{\cH}{\mathcal H}

\newcommand{\argmin}{{\rm argmin}\kern 0.12em}
\newcommand{\dotp}[2]{\langle #1,\,#2 \rangle}
\def\a{\alpha}

\def\b{\beta}
\def\d{\delta}

\def\g{\gamma}
\def\m{\mu}

\def\l{\lambda}
\def\<{\langle}
\def\>{\rangle}
\def\r{\rho}

\def\p{\partial}
\def\n{\nabla}
\def\prox{\hbox{prox}}

\theoremstyle{thmstyleone}%
\newtheorem{theorem}{Theorem}
\newtheorem{proposition}[theorem]{Proposition}%
\newtheorem{corollary}[theorem]{Corollary}%
\newtheorem{lemma}[theorem]{Lemma}%

\theoremstyle{thmstyletwo}%
\newtheorem{example}{Example}%
\newtheorem{remark}{Remark}%

\theoremstyle{thmstylethree}%

\PassOptionsToPackage{normalem}{ulem}
\usepackage{mathtools}
\DeclarePairedDelimiter\norm{\lVert}{\rVert}%

\begin{document}

\title[temporal scaling  and Tikhonov approximation of a first-order dynamical system]{Strong  convergence towards the minimum norm solution via temporal scaling and Tikhonov approximation of a first-order dynamical system}
\author[1]{\fnm{Akram Chahid} \sur{Bagy}}\email{akram.bagy@gmail.com}
\author[1]{\fnm{Zaki} \sur{Chbani}}\email{chbaniz@uca.ac.ma}
\author*[1]{\fnm{Hassan} \sur{Riahi}}\email{h-riahi@uca.ac.ma}
\affil[1]{\orgdiv{Department of Mathematics}, \orgname{Faculty of Sciences, Cadi Ayyad University}, \orgaddress{\city{Marrakesh}, \postcode{40000},  \country{Morocco}}}

\abstract{
Given a proper convex lower semicontinuous function defined on a Hilbert space and whose solution set is supposed  nonempty. For attaining a global minimizer when this convex function is continuously differentiable,
we approach it  by a first-order continuous dynamical system with 
a time rescaling parameter
and a Tikhonov regularization term. We show, along the generated trajectories, fast convergence of values, fast convergence of gradients towards origin and strong convergence towards the minimum norm element of the solution set.
These convergence rates now depend on the time rescaling parameter, and thus improve existing results by choosing this parameter appropriately. 
%
%
The obtained results illustrate, via particular cases on the choice of  the time rescaling parameter, good performances of
the proposed continuous method and the wide range of applications they can address.
Numerical  illustrations for continuous  example is provided to confirm the theoretical results.
}
%
\maketitle

\large

\section{Introduction}\label{sec1}

In a Hilbert setting  $\mathcal{H}$, we denote by $\langle\cdot ,\cdot\rangle$ and  $\Vert \cdot\Vert$ the associate inner product and norm respectively. Let $f:\,\cH\rightarrow \mathbb R$ be a  continuously differentiable convex function.
Consider the following  optimization problem 
\begin{equation}\label{minf}
\tag*{$(\mathcal{P})$} \min \left\lbrace f(x):\, x\in \mathcal{H}\right\rbrace ,
\end{equation}
 whose global solution set $S:=   \argmin_{\mathcal{H}} f $ is assume to be  nonempty.
\\
The main aim of this paper is to approach the Cauchy first order dynamical system 
 \begin{equation}\label{FODE}
 \dot{x}(t) + \nabla f (x(t))= 0.
\end{equation}
via a high-value penalty parameter $\b(t)$ and  a Tikhonov regularizing term $\frac{c}{\b(t)}$:
\begin{equation}\label{systeme continue}
\dot{x}(t)+\b(t) \n f(x(t))+c x(t) =0.
\end{equation}
Then, via  Lyapunov analysis, we obtain convergence results ensuring fast convergence of values, fast convergence of gradients towards zero and strong convergence of the corresponding solution $x(t)$ towards the minimum norm element of $S$.\\
\if{

The goal of this paper is to construct a  suitable  trajectory

sequence $(x_k)$
in order to jointly obtain the fast rate of convergence for  the objective function $f(x_k)$ to the optimal  value $f^\star=\min_{\mathcal{H}}f$, the fast rate norm convergence  to zero of a selected gradient in the subdifferential  $\partial f(x_k)$, and also the strong convergence of the iterates $x_k$ to $x^\star$ the element of minimum norm of $S$.\\
As a guide in our study, we first rely on the asymptotic behaviour  of trajectories

of  a
suitable continuous dynamical system.
Let's first recall some classic results concerning  the continuous steepest descent method
 \begin{equation}\label{FODE}
 \dot{x}(t) + \nabla f (x(t))= 0.
\end{equation}
This gradient  method goes back to Cauchy (1847),  while the proximal algorithm was first introduced by Martinet \cite{martinet} (1970), and then developed   by Rockafellar \cite{Rock} who extended it to solve monotone inclusions.  One can consult \cite{BaCo,PaBo,Pey,Pey_Sor,Polyak2} for a recent account on the proximal methods, that play a central role in nonsmooth optimization as a basic block of many splitting algorithms.
\\
}\fi 
Let us recall that  recent research axes have focused on coupling first-order time-gradient systems with a Tikhonov approximation whose coefficient tends asymptotically to zero.
By solving a general  ill posed problem $b= Bx$ in the sense of Hadamard, Tikhonov proposed the new method  which he called a method of "regularization", see \cite{Tikhonov1,Tikhonov2}. This method, which has been developed  in \cite{Morozov,TikhonovLY} and references therein, consists  first in solving the well-posed problem $b=Bx_\epsilon +\epsilon x_\epsilon$, and then in converging $x_\epsilon$ towards a selected point $\bar x$ (as $\epsilon\rightarrow 0$) which verifies $B\bar x=b$.
\\
The minimization of the function $ \varphi _t(x):=f(x)+ \dfrac{c}{2\beta (t)}\norm{x} ^2 $, where  $c$ is a positive real number and $\beta (t) $ goes to $+\infty$ as $t\rightarrow +\infty$, can be seen as a penalization of the problem of minimizing  the objective function $\frac12\|\cdot\|^2$ under the constraint $x\in \hbox{argmin} f$. This is also a two level hierarchical minimization problem, see \cite{cabo05,attcz10,attcz17,ccr00}. 
\\
Knowing that $ \psi_t:=\beta (t)\varphi _t=\frac{c}{2}\norm{\cdot} ^2+\beta (t)f$ cross towards the  proper convex lower semicontinuous function $\bar\psi:=\frac{c}{2}\norm{\cdot} ^2+\iota_{\hbox{argmin} f}$ as $t\rightarrow+\infty$ if $\iota_{C}$ is the indicator function of the set $C$, i.e., $\iota_{C} (x) = 0$ for $x\in C$, and $+\infty$ outwards, this monotone convergence is proved as a variational convergence  and then as $t\rightarrow+\infty$ (see \cite[Theorems 3.20, 3.66]{att83}) the corresponding weak$\times$strong or strong $\times$weak graph convergence of the associated subdifferential operators: $\nabla \psi_t{\displaystyle\longrightarrow^G} \partial\bar\psi$.  As $t\rightarrow +\infty$, suppose that the unique minimizer $z_t$ of $\psi_t$ weakly converges to some $\bar z$, then $(z_t,0)$ weak$\times$strong converges to $(\bar z,0)$ in the graph of $\bar\psi$; this can be explained as  
\[
0\in  \partial\bar\psi(\bar z)= \partial\left(\dfrac{c}{2}\norm{\cdot} ^2+\iota_{\hbox{argmin} f}\right)(\bar z) = c\bar z +  \partial(\iota_{\hbox{argmin} f})(\bar z).
\]
 The final equality is due to continuity of the convex function $\frac{c}{2}\norm{\cdot} ^2$ at some point in the nonempty set ${\hbox{argmin} f}$, which is the effective domaine of the  convex lower semicontinuous function $\iota_{\hbox{argmin} f}$.
\\
Suppose $f$ is nondifferentiable, then asymptotical behaviour, for $t\rightarrow+\infty$, can be explained as the steepest descent dynamical system 
\begin{equation}\label{systeme continue2}
0\in \dot{x}(t)  +\partial\left(\dfrac{c}{2}\norm{\cdot} ^2+\beta (t)f\right)(x(t))=  \dot{x}(t)  +cx(t) +\beta (t)\partial f(x(t)).
\end{equation}
 An abundant literature has been devoted to the asymptotic hierarchical minimization property which results from the introduction of a vanishing viscosity term (in our context the Tikhonov approximation) in gradient-like dynamics. For first-order gradient systems and subdifferential inclusions, see \cite{Att2,AttCom}. In \cite{AttCom},  Attouch and Cominetti coupled the dynamic steepest descent method and a Tikhonov regularization term
$$
\dot x(t)+\partial f(x(t))+\epsilon (t)x(t)\ni 0.
$$
   The striking point of their analysis is the strong convergence of the trajectory $x(t)$ when the regularization parameter $\epsilon (t)$ tends to zero with a sufficiently slow rate of convergence $\epsilon \not\in L^1(\mathbb R_+, \mathbb R)$. Then the strong limit is the minimum norm element of  $\hbox{agrmin }f$. However, if $\epsilon (t) = 0$ we can only expect a weak convergence of the induced trajectory $x(t)$.
 Attouch and Czarnecki in \cite[Theorem 3.1]{attcz10} studied the asymptotic behaviour, as time variable $t$ goes to $+\infty$, of a general nonautonomous first order dynamical system 
$$
0\in \dot{x}(t)  +\partial \varphi (x(t)) + \beta(t)\partial f(x(t)),
$$
and proved weak convergence of $x(t)$ to some $\bar x$ in $\hbox{argmin}\varphi $ on ${\hbox{argmin} f}$ that satisfies $0\in \partial(\varphi +\iota_{\hbox{argmin} f})(\bar x)$. This can be translated for $\varphi = \frac12\|\cdot\|^2$ 
 to $\bar x$ is the minimal norm solution  of  \ref{minf}. This can be considered as a combination of two techniques: the time scaling of a damped inertial gradient system (see \cite{ACR1, ABCR,ACR2, ACRA}), and the Tikhonov regularization of such systems (see \cite{BCR1, ABCR-JDE} and related references).\\
Our first-order approach in \eqref{systeme continue} derives the following fast convergence results:
\begin{eqnarray}\label{conv-rap:sontsyst}
&&f(x(t)) - \min_{\mathcal H}f = o\left( \dfrac{1}{\b (t)}\right)\hbox{ and }\norm{\nabla f(x(t))}^2 	= o\left( \dfrac{1}{\b(t)}\right).
\end{eqnarray}
%
 
In the context of non-autonomous dissipative dynamic systems, reparameterization in time is a simple and universal
means to accelerate the convergence of trajectories. This is where the coefficient
$\beta(t)$ comes in as a factor of $\nabla f(x(t))$:
\begin{equation}\label{equation2}
\ddot{x}(t) +  \frac{\alpha}{t} \, \dot{x}(t) + \beta (t) \nabla f(x(t))    =0,
\end{equation}

 In  \cite{ACR1,ACR2,bk1},
the authors  proved that under appropriate conditions on $\alpha$ and $\beta(t)$, 
$f(x(t)) -\min f =  \mathcal{O} \left( \frac1{t^2\beta(t)} \right)$, hence an improvement of the convergence rate  for the values
is reached by taking $\beta (t)\rightarrow +\infty$ as $t\rightarrow +\infty$, however strong convergence of $x(t)$ was omitted due to lack of control. Similar results were established in \cite{ACFR} when the inertial system \eqref{equation2} is with a Hessian-driven damping.

In the later papers  \cite{BCR1,BCR2}, we considered a similar system as \eqref{equation2} without Nesterov's acceleration parameter $\frac1t$ and with the Tikhonov regularizing term $\frac{c}{\beta (t)}\|\cdot\|^2$ to the convex function $f$:
\begin{equation}\label{equation}
 \ddot{x}(t) +  \alpha \, \dot{x}(t)+  \beta (t)\left\lbrace   \nabla f(x(t)) + \dfrac{c}{\beta(t)} x(t)\right\rbrace     =0.
\end{equation}
This is exactly a second-order variation in time of our system \eqref{systeme continue}. Thus this system \eqref{equation} is more computationally expensive than \eqref{systeme continue},  although it maintains the same convergence rates \eqref{conv-rap:sontsyst} where conditions on $\b(t)$ are wider than those imposed in \cite{ACR1} and \cite{ACR2}.

As a remarkable method for improving convergence rates by moving from a first order time system to a second order one, we can cite the recent paper by Attouch et all \cite{abn1} where the authors use the Fast convex optimization via closed-loop time scaling and averaging technical. Unfortunately, this method is not valid when the first order time system is already governed by a time rescaling parameter.

\if{
In the fifth and sixth Sections of this paper, the continuous aspect of \eqref{equation2} makes it possible to emerge simpler and less technical proofs and thus to better schematize the proof of the results associated with the algorithms generated by the proposed discretizations.

So, to attain a solution of the problem $(\mathcal{P})$ for nonsmooth convex function $f$, we consider the following implicit  discretization of the set-valued dynamical system 
$
0\in \dot{x}(t)+\b(t) \left(\partial f+\frac{c}{\b(t)}I\right) (x(t)) :
$
\begin{equation}\label{disc syst }
0\in \dot x_k + \dfrac{\b_k}{d} \left(\p f+\dfrac{1-d}{\b_k}\right)(x_{k+1}) = 
x_{k+1}-x_k+\dfrac{\b_k}{d} \p f(x_{k+1}) + \dfrac{1-d}{d}x_{k+1},
\end{equation}
where  $\b_k=d\beta(k)$ and $ d={1-c}\in ]0,1[$.
\\
Recall that the proximal operator can be formulated as follows $\hbox{prox}_{\b f}( x):=\left( I+\b \p f \right)^{-1}(x)$, then iteration \eqref{disc syst } can be reformulated as 
the "Proximal Algorithm with Tikhonov Regularization" :
\begin{equation}\label{proximal-algorithm}
\tag*{{\rm (PATiR)}}
x_{k+1}=\hbox{prox}_{\b_k f}(d x_k).
\end{equation}
Remark that this algorithm  \ref{proximal-algorithm} can be seen as a Forward-Backward iteration for a two-level hierarchical minimization problem:
\begin{center}$
\hbox{prox}_{\b_k f}( d x_k) = \hbox{prox}_{\b_k f}( x_k-c x_k) = \hbox{prox}_{\b_k f}\left( x_k-\b_k\nabla(\frac{c}{2\b_k}\|\cdot\|^2) (x_k)\right).
$\end{center}
Thus, we adhere to  establish similar proposals as fast convergence of values and  gradients towards zero and strong convergence   of the proximal algorithm \ref{proximal-algorithm} to the metric projection $x^*$ of the origine on the solution set $\hbox{agrmin }f$ under  suitable  conditions on the sequence $(\b_k)$ going to infinity.\\
%
Comparing our algorithm \ref{proximal-algorithm} and the one proposed by L\'aszl\'o in  \cite{Las23} as an implicit discretization of the dynamical system $\ddot x(t)+ \frac{\alpha}{t^q} \dot x(t)+\nabla f(x(t))+ \frac{c}{t^p} x(t)=0$ (see \cite{Las23A}), we justify the best convergences generated by \ref{proximal-algorithm}.
}\fi

The organization of the rest of the paper is as follows. In Section 2  we first recall basic facts concerning Tikhonov approximation.  
Then, under an appropriate setting of the parameters and based on Lyapunov' s analysis,  show in the main result (Theorem \ref{Th1})  that the trajectories provide jointly  fast convergence of values, fast convergence of gradients towards zero, and strong convergence to the minimum norm minimizer. 
In Section 3, we apply these results to two particular cases of tthe coefficient
$\beta(t)$. Section 4 is devoted a numerical example.
 Finally, in the last section we discuss the contribution of this paper as well as  the extension of these results to two level hierarchical minimization problems.

\section{Convergence rates for the implemented continuous system}\label{sec2}
Return to the differential equation \eqref{systeme continue} :
\begin{equation*}
\dot{x}(t)+\b(t) \n f(x(t))+c x(t) =0.
\end{equation*}
We suppose the following conditions:
\begin{equation*}\label{H01}
\tag*{$(\mathbf{H}_f)$}\left \{
\begin{array}{lll}
(i)&& f \textit{ is convex and  differentiable on } \cH,\\\\
(ii)&& S := \argmin \, f \neq \emptyset\, , \\\\
(iii)&& \nabla f  \textit{  Lipschitz continuous on  }\mathcal{H},
\end{array}
\right.
\end{equation*}
and, there exists $\mu\in (0,c)$ such that, for $t> t_0$
\begin{equation*}\label{H0}
\tag*{$(\mathbf{H}_\beta )$} \;
\left \{
\begin{array}{lll}
(i)&& \beta (t) \hbox{ is a positive,}\; \mathcal{C}^2\, \hbox{ function with }  \dot{\beta}(t)\neq 0,\\\\
(ii)&&
\dfrac{\dot{\b}(t)}{\b(t)}  \leq c-\m  ,\\
(iii)&& 
{	\underset{t\rightarrow +\infty}{\limsup} \,\dfrac{  1 + \dfrac{\dot{\beta } (t)}{(c-\m)\beta  (t)}}{\mu + \dfrac{\ddot{\beta } (t)}{{\dot\beta} (t)} - \dfrac{\dot{\b}(t)}{\b(t)}} <+\infty  .}
\end{array}
\right.\hspace{6mm}
\end{equation*}
Let us denote by $x^*$  the minimum norm element  of  $S$, and introduce the energy function $E(t)$  that  is defined on $[t_0, +\infty[$ by
\begin{equation}\label{3}
E(t) := 
\b(t)\left(\varphi_{t}(x(t))-\varphi_{t}(y(t))\right) +\dfrac{c}{2}\|x(t)-y(t)\|^{2}
\end{equation}
where $x(t)$ is a solution of \eqref{systeme continue}, $y(t) = \argmin_{\cH}{\varphi}_{t}$    and
$
\varphi_t (x) := f(x) + \dfrac{c}{2\b(t)} \|x\|^2.
$\\
\begin{lemma}\label{lemma 2} 
The following properties on $y(t)$ are satisfied:
\begin{itemize}
\item[i)] For all $t\geq t_0,$ $\dfrac{d}{dt} \left( \varphi_t (y(t)) \right)= \dfrac{-c \dot{\b}(t)}{\b^2(t)} \|y(t)\|^2 $.
\item[ii)] For almost every $t \geq t_0$,
$\| \dot y(t)\| \leq  \dfrac{\dot{\b}(t)}{\b(t)}\|y(t)\|.$
\end{itemize}
\end{lemma}
\begin{proof}
The proof of this lemma is similar to that of \cite[Lemma 2]{ABCR-JDE}, by substituting $ \varepsilon(t)$ with $\frac{c}{\b(t)}$ and $x_{\varepsilon(t)}$ with $y(t)$.
\end{proof}
We remark that \eqref{systeme continue} becomes 
\begin{equation}\label{systeme continue1}
\dot{x}(t)+\b(t) \n \varphi_{t}(x(t)) =0.
\end{equation}
and, under the initial condition $x(t_0)=x_0\in \cH$, it admits a unique solution. 
\begin{theorem}\label{Th1}
Let $f: \mathcal{H} \rightarrow \mathbb{R}$ be a convex function satisfying condition \ref{H01} and $x: [t_0 , +\infty[ \rightarrow \mathcal{H}  $ be a solution of the system \eqref{systeme continue}.\\
If $\beta(t)$ satisfies \ref{H0}, then there exists $t_1 \geq t_0$ such that, for  $t \geq t_1$  
\begin{eqnarray}
\label{eq:8-1}&&f(x(t)) - \min_{\mathcal H}f = \mathcal{O}\left( \dfrac{1}{\b (t)}\right);\\
\label{eq:8-2}&& \|x(t) - y(t)\|^2 			= \mathcal{O}\left( \frac{ 1}{e^{\mu t}} + \dfrac{\dot{\b}(t)}{\b(t)}\right);\\
\label{eq:8-3}&&\norm{\nabla f(x(t))}^2 	= \mathcal{O}\left( \dfrac{1}{\b(t)}\right).
\end{eqnarray}
Suppose moreover that $\displaystyle\lim_{t\rightarrow +\infty}\dfrac{\dot{\b}(t)}{\b(t)}=0$, then $x(t)$ strongly converges to $x^*$, and 
\begin{eqnarray}
\label{eq:8-4}&& f(x(t)) - \min_{\mathcal H}f = o\left( \dfrac{1}{\b (t)}\right)\hbox{ and }\norm{\nabla f(x(t))}^2 	= o\left( \dfrac{1}{\b(t)}\right).
\end{eqnarray}
\end{theorem}
\begin{proof}
We have defined $E(t)$ in \eqref{3} by
$E(t)=A(t) + B(t)$ where 
$$
A(t)=\beta (t) \left( \varphi _t (x(t))- \varphi _t (y(t)) \right)   \hbox{and } B(t)=\dfrac{c}{2} \norm {x(t)-y(t)}^2.
$$
So, we need to fit a first  order differential inequality on $E(t)$ in order to be able to get for a positive constant $M$ and $t_1\geq t_0$ the desired upper bound 
\begin{equation}\label{estim E}
		E(t) \le  \frac{ E(t_1)e^{\mu t_1}}{e^{\mu t}} + \frac{c M\norm{x^*}^2 }{2} \dfrac{\dot{\b}(t)}{\b(t)} \;\;\hbox{ for each }t\geq t_1.
\end{equation}
 To calculate the derivative of the functions  $A(t)$ and $B(t)$, we remark that the mapping $t\mapsto y (t)$ is absolutely continuous (indeed locally Lipschitz), see \cite[section VIII.2]{Brezis2}. We conclude that $y(t)$ is  almost everywhere  differentiable on $[t_0,+\infty[$, and then using differentiability of $f$, we also deduce  almost everywhere  differentiability of $A(t)$ and $B(t)$ on the same intervalle.
Thus, we have
$$
\dot{A}(t)=\dot{\beta} (t)\left( \varphi _t (x(t))- \varphi _t (y(t)) \right) + \beta (t) \left(  \dfrac{d}{dt} \left(\varphi _t (x(t)) \right) - \dfrac{d}{dt} \left(\varphi _t (y(t)) \right) \right) .
$$
Using the derivative calculation and \eqref{systeme continue1}, we remark that
\begin{equation*}
\begin{array}{lll}
\dfrac{d}{dt} \left(\varphi _t (x(t)) \right) &=& \dfrac{d}{dt}\left( f(x(t))+\dfrac{c}{2\beta (t)}\norm{x(t)}^2 \right) \\\\
&=& \large\langle \dot{x}(t)\, , \, \nabla f(x(t))  \large\rangle - \dfrac{c\dot{\beta } (t)}{2\beta ^2 (t)} \norm{x(t)}^2 +\dfrac{c}{\beta (t)} \large\langle \dot{x}(t) \, , \, x(t) \large\rangle \\\\
&=& \large\langle \dot{x}(t)\, , \, \nabla f(x(t))+ \dfrac{c}{\beta (t)} x(t)  \large\rangle - \dfrac{c\dot{\beta } (t)}{2\beta ^2 (t)} \norm{x(t)}^2  \\\\
&=& \large\langle \dot{x}(t)\, , \, \nabla \varphi _t (x(t))  \large\rangle- \dfrac{c\dot{\beta } (t)}{2\beta ^2 (t)} \norm{x(t)}^2 \\
&=& -\beta(t)\left\|  \nabla \varphi _t (x(t))  \right\|^2- \dfrac{c\dot{\beta } (t)}{2\beta ^2 (t)} \norm{x(t)}^2  
\end{array}
\end{equation*}
On the other hand, according to  Lemma \ref{lemma 2} (i), we have
 $$ 
 \dfrac{d}{dt} \left(\varphi _t (y (t)) \right) =  \dfrac{-c\dot{\beta}(t)}{2\beta ^2 (t)}\norm{y(t)}^2 .
 $$
Thus
\begin{equation}\label{16}
\dot{A}(t)= \dot{\beta} (t)\left( \varphi _t (x(t))- \varphi _t (y(t)) \right)   -
\beta^2 (t) \left[  \left\|  \nabla \varphi _t (x(t))  \right\|^2 + \dfrac{c\dot{\beta } (t)}{2\beta ^3 (t)}\left(  \norm{x(t)}^2- \norm{y(t)}^2 \right) \right] .
\end{equation}
We also have
\begin{equation*}
\begin{array}{lll}
\dot{B} (t) &=&   c\large\langle x(t) - y(t)  \, , \,  \dot{x}(t)- \dot y(t)  \large\rangle\\
&=&  - c\large\langle x(t) - y(t)  \, , \, \beta (t) \nabla \varphi _t (x(t)) + \dot y(t) \large\rangle .
\end{array}
\end{equation*}
Noting that for every $u,v\in \mathcal{H},\, \lambda >0$, $\langle u,v\rangle\leq \frac{\lambda}2\| u\|^2+\frac1{2\lambda}\|v\|^2$,  we get
\begin{equation*} 
- c \large\langle  x(t) - y(t) \, , \, \dot y(t) \large\rangle \leq \dfrac{c}{2}\left( {\lambda}\norm{ x(t) - y(t)}^2 + \frac1{\lambda}\norm{ \dot y(t)}^2 \right).
\end{equation*}
Using  Lemma \ref{lemma 2} (ii), we  obtain  $ \norm{ \dot y(t)} \leq \dfrac{\dot{\beta}(t)}{\beta(t)} \norm{y (t)}\leq \dfrac{\dot{\beta}(t)}{\beta(t)} \norm{x^*}$; we conclude 
\begin{equation}\label{8}
- c \large\langle  x(t) - y(t) \, , \, \dot y(t) \large\rangle \leq \dfrac{c}{2}\left( {\lambda}\norm{ x(t) - y(t)}^2 +  \dfrac{\dot{\beta}^2(t)}{\lambda\beta^2(t)} \norm{x^*}^2 \right) .
\end{equation}
Observe that for $c>0$ the function  $\varphi _t$ is $\frac{c}{\b(t)}$-strongly convex,  then we have
\begin{equation*}
 - c \beta (t)\large\langle \nabla \varphi _t (x(t))  \, , \,   x(t) - y(t)  \large\rangle  \leq  -  c \beta (t) \left( \varphi _t (x(t))-  \varphi _t \left(y(t)\right) \right) - \dfrac{c^2}{2} \norm{  x(t) - y(t) }^2 .
\end{equation*}
Thus
\begin{equation}\label{9}
\begin{array}{lll}
\dot{B}(t) &\leq&    \dfrac{c}{2}\left( \lambda\norm{ x(t) - y(t)}^2 +  \dfrac{\dot{\beta}^2(t)}{\lambda\beta^2(t)} \norm{x^*}^2 \right) \\
	&&-  c \beta (t) \left( \varphi _t (x(t))-  \varphi _t \left(y(t)\right) \right) - \dfrac{c^2}{2} \norm{  x(t) - y(t) }^2 \\
&=&  \dfrac{(\lambda-c)c}{2} \norm {x(t)-y(t)}^2  + \dfrac{c\dot{\beta}^2(t)}{2\lambda\beta^2(t)} \norm{x^*}^2  -  c \beta (t) \left( \varphi _t (x(t))-  \varphi _t \left(y(t)\right) \right) .
\end{array}
\end{equation}
Due to the inequalities \eqref{16} and \eqref{9}, we get
\begin{equation}\label{10}
\begin{array}{lll}
\dot{E}(t) &\leq &  \dot{\beta} (t)\left( \varphi _t (x(t))- \varphi _t (y(t)) \right)   -
\beta^2 (t) \left[  \left\|  \nabla \varphi _t (x(t))  \right\|^2 + \dfrac{c\dot{\beta } (t)}{2\beta ^3 (t)}\left(  \norm{x(t)}^2- \norm{y(t)}^2 \right) \right] \\
	& & +  \dfrac{(\lambda-c)c}{2} \norm {x(t)-y(t)}^2  + \dfrac{c\dot{\beta}^2(t)}{2\lambda\beta^2(t)} \norm{x^*}^2  -  c \beta (t) \left( \varphi _t (x(t))-  \varphi _t \left(y(t)\right) \right)  \\
	& = & \left(  \dot{\beta} (t)-c\beta (t)\right) \left( \varphi _t (x(t))- \varphi _t (y(t)) \right)   -
\beta^2 (t)  \left\|  \nabla \varphi _t (x(t))  \right\|^2 \\\\
&\,\,&- \dfrac{c\dot{\beta } (t)}{2\beta  (t)}\left(  \norm{x(t)}^2- \norm{y(t)}^2 \right)  +  \dfrac{(\lambda-c)c}{2} \norm {x(t)-y(t)}^2  + \dfrac{c\dot{\beta}^2(t)}{2\lambda\beta^2(t)} \norm{x^*}^2 \\
\end{array}
\end{equation}
and thus,  for all $t \geq t_0$ we have
$$
\begin{array}{lll}
\dfrac{d}{dt} \left[e^{\mu t} E(t) \right] &=& \left( \dot{E}(t)+ \mu E(t) \right)e^{\mu t} \\
	&\leq&  \Big[  \left(  \dot{\beta} (t)-(c-\m)\beta (t)\right) \left( \varphi _t (x(t))- \varphi _t (y(t)) \right)   -
\beta^2 (t)  \left\|  \nabla \varphi _t (x(t))  \right\|^2 \\
&\,\,&
- \dfrac{c\dot{\beta } (t)}{2\beta  (t)}\left(  \norm{x(t)}^2- \norm{y(t)}^2 \right)   +  \dfrac{(\lambda+\mu-c)c}{2} \norm {x(t)-y(t)}^2  \\
&&  + \dfrac{c\dot{\beta}^2(t)}{2\lambda\beta^2(t)} \norm{x^*}^2   \Big] e^{\mu t} .
\end{array}
$$
Using assumption \ref{H0} (ii) and choosing $\lambda=c-\m>0$, we conclude that  for $t\geq t_0$ large enough
\begin{equation}\label{11}
\dfrac{d}{dt} \left[e^{\mu t} E(t) \right] \leq   \dfrac{c\dot{\beta } (t)}{2\beta  (t)}\left(   \norm{y(t)}^2 + \dfrac{\dot{\beta } (t)}{(c-\m)\beta  (t)} \norm{x^*}^2\right)  e^{\mu t} .
\end{equation}
Since $\; \norm{y(t)} \leq \norm{x^*}\;  $  (see Lemma \ref{lemma 2} (ii)), we conclude
\begin{equation}\label{13}
\dfrac{d}{dt} \left[e^{\mu t} E(t) \right] \leq   \dfrac{c\dot{\beta } (t)}{2\beta  (t)}\left(  1 + \dfrac{\dot{\beta } (t)}{(c-\m)\beta  (t)}\right)  e^{\mu t} \norm{x^*}^2.
\end{equation}

	Now return to condition \ref{H0}(iii), which means
	$$
	\underset{t\rightarrow +\infty}{\limsup} \,\dfrac{\dfrac{\dot{\beta } (t)}{\beta  (t)} \left(  1 + \dfrac{\dot{\beta } (t)}{(c-\m)\beta  (t)}\right)e^{\mu t } }{ \dfrac{d}{dt}\left(\dfrac{\dot{\b}(t)}{\b(t)}e^{\mu t }\right)} <+\infty,
	$$
we obtain existence of $t_1\geq t_0$ and $M>0$ such that for each $t\geq t_1$
	\begin{equation} \label{e14}
	\dfrac{d}{dt} \left[e^{\mu t} E(t) \right] \leq \frac{cM \norm{x^*}^2 }{2} \dfrac{d}{dt}\left(\dfrac{\dot{\b}(t)}{\b(t)}e^{\mu t }\right) .
	\end{equation}
	By integrating on $[t_1,t]$, we get  
	\begin{eqnarray}
		E(t) &\le&  \frac{ E(t_1)e^{\mu t_1}}{e^{\mu t}} + \frac{cM\| x^* \|^2 }{2e^{\mu t}} \int_{t_1}^{t} 
		\dfrac{d}{ds} \left(\dfrac{\dot{\b}(s)}{\b(s)}e^{\mu s }\right)ds \nonumber\\
		 &\leq&  \frac{ E(t_1)e^{\mu t_1}}{e^{\mu t}} + \frac{c M\norm{x^*}^2 }{2} \dfrac{\dot{\b}(t)}{\b(t)} .   \label{Ee1}
	\end{eqnarray}
 According to the definition of $ \varphi_{t}$,  we have 
 \begin{equation*}
\begin{array}{lll}
f(x(t)) &-& \min_{\mathcal H}f	=  \varphi_{t}(x(t))-\varphi_{t}(x^*)+\dfrac{c}{2\b (t)}\left(\|x^*\|^{2}-\|x(t)\|^{2}\right)  \\ 
	& = & \left[\varphi_{t}(x(t))-\varphi_{t}(y(t))\right]+ [\underbrace{\varphi_{t}(y(t))-\varphi_{t}(x^*)}_{\leq 0} ]+\dfrac{c}{2\b (t)}\left(\|x^*\|^{2}-\|x(t)\|^{2}\right)\\
	& \leq  &\varphi_{t}(x(t))-\varphi_{t}(y(t))+\dfrac{c}{2\b (t)}\left(\|x^*\|^{2}-\|x(t)\|^{2}\right).
\end{array}
\end{equation*}
By definition of $E(t)$,\;  
$$
\varphi_{t}(x(t))-\varphi_{t}(y(t)) \leq \dfrac{E(t)}{\b (t)},
$$
 which, combined with the above inequality and \eqref{Ee1}, gives 
 \begin{equation}\label{eq:17a}
f(x(t)) - \min_{\mathcal H}f \leq \dfrac{1}{\b (t)}\left(\frac{ E(t_1)e^{\mu t_1}}{e^{\mu t}} + \frac{c M\norm{x^*}^2 }{2} \dfrac{\dot{\b}(t)}{\b(t)}+\frac{c}{2}\left(\|x^*\|^{2}-\|x(t)\|^{2}\right)\right). 
\end{equation}
This means, according to \ref{H0}(ii), that 
$f$ satisfies \eqref{eq:8-1}.
\\
Also, 
using the definition of $E(t)$, we get
$$
E(t) \geq \frac{c}{2}  \|x(t) - y(t)\|^2,
$$
which gives 
 \begin{equation*}
 \|x(t) - y(t)\|^2 \leq \frac{2}{c}  E(t)  \leq  \frac{2}{c}\left(\frac{ E(t_1)e^{\mu t_1}}{e^{\mu t}} + \frac{c M\norm{x^*}^2 }{2} \dfrac{\dot{\b}(t)}{\b(t)}\right),
\end{equation*}
and therefore \eqref{eq:8-2} is satisfied.
\\
Return to \ref{H0}(ii), we justify $(x(t))$ is bounded, and then combining \eqref{eq:17a} and Lemma \ref{ext_descent_lemma} we deduce  existence of $L>0$ such that for $t$ large enough
\begin{eqnarray}\label{eq:18b}
 \| \nabla f (x)\|^2  &\leq& 2L(f(x) - \min_{\cH} f) \nonumber\\
 &\leq&  \dfrac{2L}{\b (t)}\left(\frac{ E(t_1)e^{\mu t_1}}{e^{\mu t}} + \frac{c \norm{x^*}^2 }{2} \dfrac{\dot{\b}(t)}{\b(t)}+\frac{c}{2}\left(\|x^*\|^{2}-\|x(t)\|^{2}\right)\right).
\end{eqnarray}
We conclude, for $t$ large enough, the estimation \eqref{eq:8-3}. \\
To ensure \eqref{eq:8-4}, we come back to the condition $\frac{\dot{\b}(t)}{\b(t)}\rightarrow 0$. Then  $\|x(t) - y(t)\|\rightarrow 0$ and strong convergence of $y(t)$ towards $x^*$ ensures that  $x(t)-x^*$ strongly converges to zero. Thus $\lim_{t\rightarrow+\infty}\left(\|x^*\|^{2}-\|x(t)\|^{2}\right)=0$, and appealing to inequalities \eqref{eq:17a}, \eqref{eq:18b}, we deduce \eqref{eq:8-4}.
\end{proof}
%
\section{Particular cases on the choice of $\b(t)$}\label{sec3}

We illustrate the theoretical conditions on $\beta(t)$ by two interesting examples:
\subsection{Case $\beta (t)=t^m\ln^pt$}

Consider the positive function $\beta (t)=t^m\ln^pt$ for $t\geq t_0>0$ and $(m,p)\in \mathbb R_+^2\setminus\{(0,0)\}$. We have
\begin{eqnarray*}
&&\dot\b(t) = t^{m-1}\ln^{p-1}t \left(m\ln t +p\right),\\
&&\ddot\b(t) = t^{m-2}\ln^{p-2}t \left(m(m-1)\ln^2 t + (2m-1)p\ln t + p(p-1)\right).\\
\end{eqnarray*}
Then, for $t$ large enough, we have the function $\frac{\dot\b(t)}{ \b(t)}$ (respectively $\frac{\ddot\b(t)}{\dot \b(t)}$) is equivalent to $\frac{p}{t\ln t}$ if $m=0$, and $\frac{p}{t}$ if $m\neq 0$ (respectively to $\frac{m-1}{t}$ if $m\neq 1$, and $\frac{1}{t\ln t}$ if $m= 1$).\\
 Thus
	$
	\underset{t\rightarrow +\infty}{\limsup} \,\frac{  1 + \frac{\dot{\beta } (t)}{(c-\m)\beta  (t)}}{\mu + \frac{\ddot{\beta } (t)}{{\dot\beta} (t)} - \frac{\dot{\b}(t)}{\b(t)}}= \dfrac1{\mu} <+\infty,
	$ for each $(m,p)\in \mathbb R_+^2\setminus\{(0,0)\}$.\\
	We conclude that condition \ref{H0} is satisfied whenever $0<\m< c$.
	\begin{corollary}\label{CorCont1}
	If $f: \mathcal{H} \rightarrow \mathbb{R}$  satisfies  \ref{H01}, $\beta(t)=t^m\ln^pt$ for $t\geq t_0>0, (m,p)\in \mathbb R_+^2\setminus\{(0,0)\}$ and $x: [t_0 , +\infty[ \rightarrow \mathcal{H}  $ is a solution of  \eqref{systeme continue}. Then
 $x(t)$ strongly converges to $x^*$, and  for  $t \geq t_0$  large enough
\begin{eqnarray}
\label{eq:20-1b}&&f(x(t)) - \min_{\mathcal H}f = o\left( \dfrac{1}{t^m\ln^pt}\right);\\
\label{eq:20-2b}&& \|x(t) - y(t)\|^2  = \left\{\begin{array}{ll} 
\mathcal{O}\left( \dfrac{1}{t\ln t}\right) & \hbox{ if  } m=0,\\
\mathcal{O}\left( \dfrac{1}{t}\right) & \hbox{ if } m\neq 0;
\end{array}\right.\\
\label{eq:20-3b}&&\norm{\nabla f(x(t))}^2 	=  o\left( \dfrac{1}{t^m\ln^pt}\right).
\end{eqnarray}
	\end{corollary}	
\subsection{Case $\beta (t)=t^me^{\gamma t^r}$}
If $\beta (t)=t^me^{\gamma t^r}$ for $t\geq t_0>0, \; m\geq 0,\; \g>0$ and $ 0<r\leq 1$, then we have
\begin{eqnarray*}
&&\dot\b(t) = t^{m-1}e^{\gamma t^r} \left(m +r\g t^r\right),\\
&&\ddot\b(t) =  t^{m-2}e^{\gamma t^r} \left(m(m-1) + (2m+r-1)r\g t^r + r^2\g^2 t^{2r} \right).\\
\end{eqnarray*}
Then, for $t$ large enough, we have the function $\frac{\dot\b(t)}{ \b(t)}$ (respectively $\frac{\ddot\b(t)}{\dot \b(t)}$) is equivalent to $\frac{r\g}{ t^{1-r}}$ (respectively to  $\frac{r\g}{ t^{1-r}}$).\\
 Thus
	$$
	\underset{t\rightarrow +\infty}{\limsup} \,\frac{  1 + \frac{\dot{\beta } (t)}{\beta  (t)}}{\mu + \frac{\ddot{\beta } (t)}{{\dot\beta} (t)} - \frac{\dot{\b}(t)}{\b(t)}}= \left\{\begin{array}{lll}
	\dfrac{1}{\mu} <+\infty &\hbox{ if }& r<1,\\
	\dfrac{1+\g}{\mu} <+\infty &\hbox{ if }& r=1.
	\end{array}\right.
	$$
	We conclude that conditions on $\b$ are satisfied whenever  $0<\m< c$.
	\begin{corollary}\label{CorCont2}
	If $f: \mathcal{H} \rightarrow \mathbb{R}$  satisfies  \ref{H01}, $\beta (t)=t^me^{\gamma t^r}$ for $t\geq t_0>0, \; m\geq 0,\; \g>0,\; 0<r\leq 1$, and $x: [t_0 , +\infty[ \rightarrow \mathcal{H}  $ is a solution of  \eqref{systeme continue}. Then
for  $t \geq t_0$  large enough
\begin{eqnarray}
\label{eq:20-1}&&f(x(t)) - \min_{\mathcal H}f = \mathcal{O}\left( \dfrac{1}{t^me^{\gamma t^r}}\right);\\
\label{eq:20-2}&& \|x(t) - y(t)\|^2  = \mathcal{O}\left( \dfrac{1}{t^{1-r}}\right);\\
\label{eq:20-3}&&\norm{\nabla f(x(t))}^2 	= \mathcal{O}\left( \dfrac{1}{t^me^{\gamma t^r}}\right).
\end{eqnarray}
When $0<r<1$, we  conclude that  $x(t)$ strongly converges to $x^*$, 
$$ 
f(x(t)) - \min_{\mathcal H}f = o\left( \dfrac{1}{\b (t)}\right)\;\hbox{ and }\;\norm{\nabla f(x(t))}^2 	= o\left( \dfrac{1}{\b(t)}\right).
$$   
	\end{corollary}
\begin{figure} 
 \includegraphics[scale=0.35]{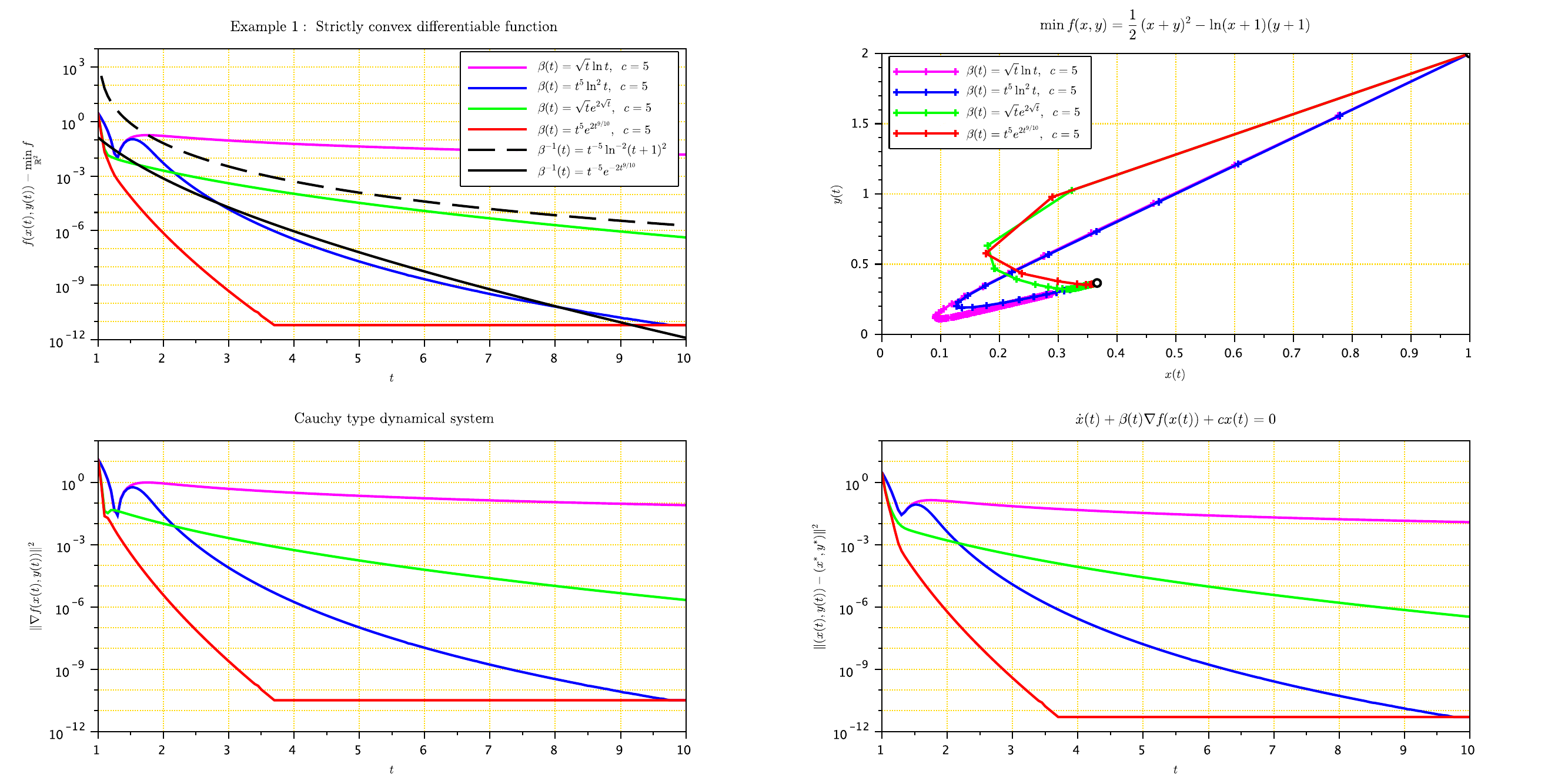}
  \caption{Evolution of convergence rates for values, trajectories and gradients when $c=5$ is fixed and $\b (t)$ changes expressions from logarithm to exponential. Note, see the figure  at the top left, that convergence rates respect our predictions in the results mentioned above.}
 \label{fig:trigs-c} 
\end{figure}

\section{Example for comparison of the convergence rate}\label{sec4}

\begin{example}\label{exemple1} 
Take $f  :  ]-1,+\infty[^{2} \to \R$ which is defined for $x=(x_1,x_2)$ by 
$$
f(x)=\frac12(x_1+x_2)^2-\ln(x_1+1)(x_2+1).
$$
\end{example}

 The function $f_1$ is strictly convex with gradient
$
\nabla f(x)=\begin{bmatrix}
x_1+x_2-\frac1{x_1+1}\\
x_1+x_2-\frac1{x_2+1}
\end{bmatrix}$ and Hessian $ \nabla^2 f(x)=\begin{bmatrix}1+\frac1{(x_1+1)^2} & 1\\1 & 1+\frac1{(x_2+1)^2}\end{bmatrix} $, so
the unique minimizer of $f$ is $x^*= \left((\sqrt {3}-1)/2,(\sqrt {3}-1)/2\right)$. 

The corresponding trajectories to the system \eqref{systeme continue} are depicted in Figure \ref{fig:trigs-c}. We note that the convergence rates of the values and gradients in this numerical test are consistent with those predicted in Corollaries \ref{CorCont1} and \ref{CorCont2}, while the convergence rates of the gradients are clearly stronger than those predicted theoretically. 
Let us consider the times-second-order systems, treated in the very recent papers: 
\[
\begin{array}{lll}
\hbox{(TRAL)}&&\dot{x}(t)+2t^2\ln^2t\nabla f\left(x(t)\right) + 5x(t)=0,\\
\hbox{(TRAE)}&&\dot{x}(t)+2t^2e^{2t^{9/10}}\nabla f\left(x(t)\right) + 5x(t)=0,\\
\hbox{(TRISAL)}&&\ddot{x}(t)+5\dot{x}(t)+2t^2\ln^2t\nabla f\left(x(t)\right) + 5x(t)=0,\hbox{ see \cite{BCR1}, }\\
\hbox{(TRISAE)}&&\ddot{x}(t)+5\dot{x}(t)+2t^2e^{2t^{4/5}}\nabla f\left(x(t)\right) + 5x(t)=0,\hbox{ see \cite{BCR1}, }\\
\hbox{(TRISG)}&&\ddot{x}(t)+5t^{-4/5}\dot{x}(t)+\nabla f\left(x(t)\right) +t^{-8/5}x(t)=0,\hbox{ see  \cite{ABCR-JDE}, } \\
\hbox{(TRISH)}&&\ddot{x}(t)+5t^{-4/5}\dot{x}(t)+\nabla f\left(x(t)\right) + 2\nabla^2 f\left(x(t)\right)\dot{x}(t) + t^{-8/5}x(t)=0, \hbox{ see  \cite{ABCR-JDE}.}
\end{array}
\]
By comparing  the two times-first-order systems, \eqref{systeme continue} where $\beta (t)$ is either equal to $2t^2\ln^2t$ or $2t^2e^{2t^{9/10}}$, with those of second order, we are surprised (see Figure \ref{fig:trigs-d}) by the better rate of convergence brought by these two reduced and inexpensive systems.
\begin{figure} 
 \includegraphics[scale=0.35]{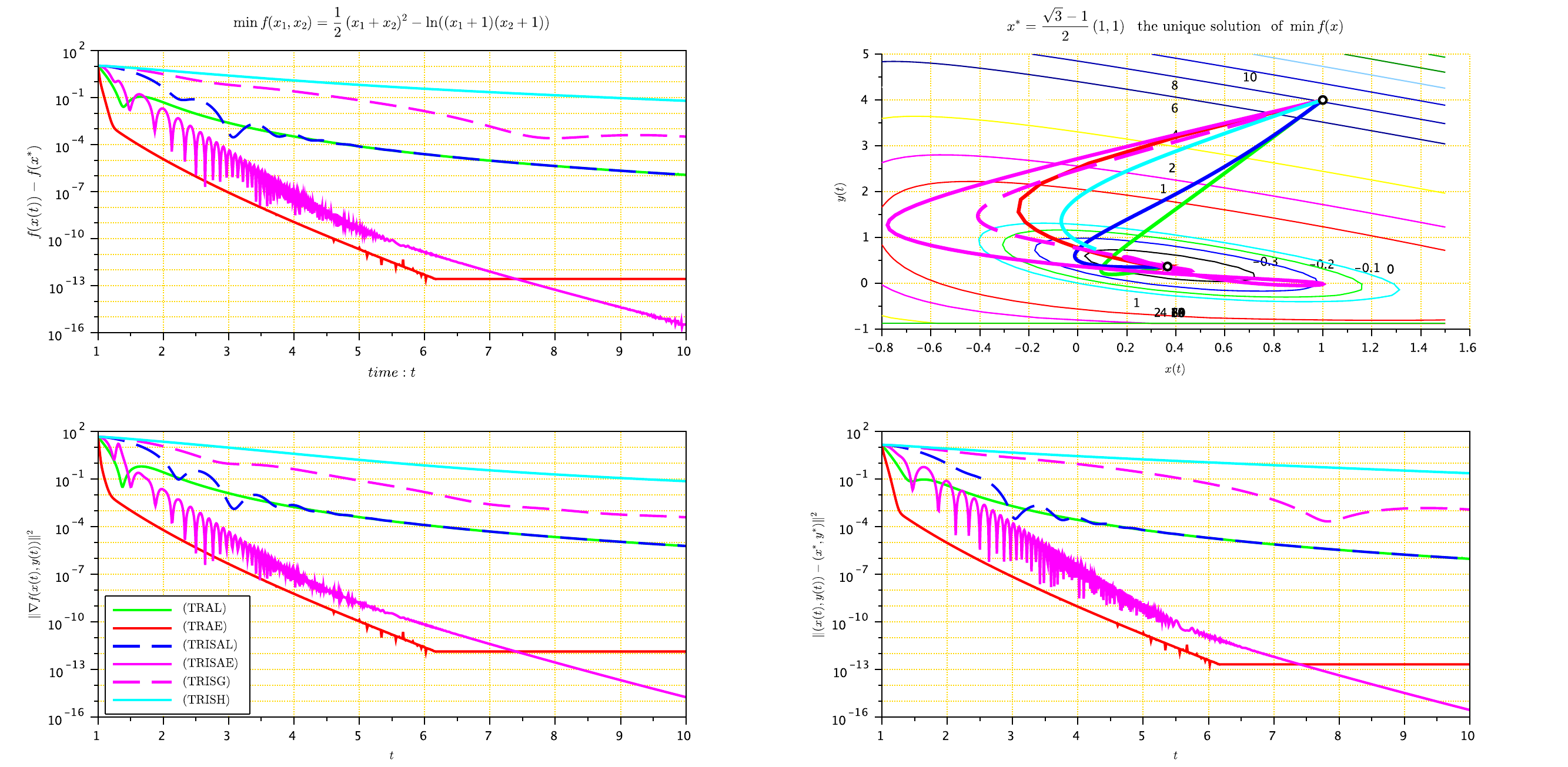}
  \caption{Here, we compare the convergence rates of values, trajectories and gradients for our first order system \eqref{continuous system} and those of the second order of recent references \cite{ABCR-JDE,BCR1}. }
 \label{fig:trigs-d} 
\end{figure}

\if{
\section{Implicit discretization for nonsmooth convex functions}\label{sec5}

Here, we suppose $f:\,\cH\rightarrow \mathbb R\cup\{ +\infty\}$ to  be a  proper lower semicontinuous convex function. 
Let us recall that an implicit discretization of the dynamical system \eqref{systeme continue} with $\varphi _k:=f+ \frac{1-d}{2\b_k}\norm{\cdot} ^2$, $\beta(k)=\b_k>0$ and $d=1-c>0$ leads to
$$
d(x_{k+1}-x_k) + \b_k   \p \varphi_k(x_{k+1}) =   x_{k+1} +\b_k\p f(x_{k+1})- dx_k \ni 0,
$$
where the time step size is equal to $1/d$.
We thus attain  the numerical scheme  \ref{proximal-algorithm} that reads  either as the Proximal method
\begin{eqnarray}
\label{alg_xk}
&&x_{k+1}= \hbox{prox}_{\frac{\b_k}{d} \varphi_k}( x_k)
\end{eqnarray}
or as the Proximal-Gradient method
\begin{eqnarray*}
&&x_{k+1}=\hbox{prox}_{\b_k f}( d x_k) = \hbox{prox}_{\b_k f}( x_k-c x_k) = \hbox{prox}_{\b_k f}\left( x_k-\b_k\nabla(\frac{c}{2\b_k}\|\cdot\|^2) (x_k)\right).
\end{eqnarray*}
We suppose the following condition:
\begin{equation*}
\tag*{$(\mathbf{H}_{\beta_k})$} 
  (\b_k)\textit{ is a nondecreasing positive sequence that satisfies }\, \underset{k\rightarrow +\infty}{\lim}\, \b_k= + \infty .
\end{equation*}
For each $k\geq k_0$, we denote by $y_{k}$ the unique minimizer of the strongly convex function
$ \varphi _k $, which means by the first order convex optimality condition
\begin{equation}\label{nn}
\p f (y_{k})+\dfrac{1-d}{\b_k} y_{k}\ni 0\; \iff \; y_k=\hbox{prox}_{\frac{1-d}{\b_k} f}(0).
\end{equation}
Let us recall (see Lemma \ref{lem-basic-c} \ref{2a},\ref{2b}) that the Tikhonov approximation curve $ k \mapsto y_{k}$ satisfies
\begin{equation}\label{x^*}
	\forall k\geq k_0, \;\; \norm{y_{k}}\leq \norm {x^*}
\hbox{ and }	\underset{k\rightarrow +\infty}{\lim} \norm{y_{k}-x^*} = 0.
\end{equation}
Now, for $\l$ a positive constant, we introduce the following  discrete energy function
\begin{equation}\label{E_k}
E_k = \b _k \left( \varphi_k (x_k)- \varphi_k (y_{k})  \right)+ \frac{\l}{2}\norm {x_k-y_{k-1}}^2, \;\; \hbox{for all }\,k\geq k_0.
\end{equation}
Before giving  this section's principal theorem, we need two very important lemmas. The first lemma relates the asymptotic behavior of the sequence $E_k$ to the convergence rate of values and iterations, the second  lemma provides a few properties of the viscosity curve $( y_{k})_k.$
\begin{lemma}\label{Lemme 1 discr}
Let $(x_k)$ be the sequence generated by the algorithm  \ref{proximal-algorithm}. Then for any $k \geq k_0$ we have:  
 \begin{equation}\label{con- val disr}
f(x_k)-\underset{\mathcal{H}}{\min} \,f \leq \dfrac{ E_k}{\b_k} + \dfrac{1-d}{2\b_k} \left(\norm{x^*}^2  -\norm{x_k}^2\right)  
\end{equation}
and
\begin{equation}\label{conv-tra discr}
\norm{x_k-y_{k}}^2 \leq \dfrac{2E_k }{1-d}.
\end{equation}
Therefore, $x_k$ converges strongly to $x^*$ as soon as $\underset{k\rightarrow +\infty}{\lim} E_k = 0$.
\end{lemma}
\begin{proof}
 Observing  the definition of $\varphi _k$, one has
\begin{eqnarray}
f(x_k)-\underset{\mathcal{H}}{\min} \,f &=& \varphi _k (x_k) - \dfrac{1-d}{2\b_k}\norm{x_k}^2 -\varphi _k (x^*) + \dfrac{1-d}{2\beta _k}\norm{x^*}^2 		\nonumber\\
&=&  \left[  \varphi _k (x_k)-  \varphi _k (x_{\b_k})  \right] +\left[ \underbrace{\varphi _k (y_{k})- \varphi _k (x^*)}_{ \leq 0}\right] 		+\dfrac{1-d}{2\beta _k}\left(\norm{x^*}^2  -\norm{x_k}^2\right)  		\nonumber\\
&\leq &  \varphi _k (x_k)-  \varphi _k (y_{k}) +\dfrac{1-d}{2\beta _k}\left(\norm{x^*}^2  -\norm{x_k}^2\right)   . \label{eq21}
\end{eqnarray}
From definition of $E_k$, we obtain
\begin{equation}\label{3d}
 \varphi _k (x_k)-  \varphi _k (y_{k}) \leq \dfrac{ E_k}{\beta _k},
\end{equation}
which, combined with \eqref{eq21}, gives \eqref{con- val disr}.\\
 Using the strong convexity of $ \varphi _k$, and $y_{k}:=\underset{\mathcal{H}}{\argmin }\, \varphi _k$, we obtain
\begin{equation*}
  \varphi _k (x_k)-  \varphi _k (y_{k})\geq \dfrac{1-d}{2\beta _k} \norm{x_k -y_{k}}^2 .
  \end{equation*}
By combining the last inequality  with \eqref{3d}, we get
\begin{equation*}
 \dfrac{ E_k}{\beta _k} \geq  \dfrac{1-d}{2\beta _k} \norm{x_k-y_{k}}^2 ,
 \end{equation*}
which implies the  inequality \eqref{conv-tra discr}.
Recall that $\underset{k\rightarrow +\infty}{\lim} \norm{y_{k}-x^*} = 0$, so according to  $\underset{k\rightarrow +\infty}{\lim} E_k = 0$, we deduce that the sequence $(x_k )$ converges strongly to $x^*$. 
\end{proof}
\begin{lemma}\label{lemma 2 discr}
For every $k\geq k_0$, the following properties are satisfied:
\begin{itemize}
\item[i)] $\varphi _k(y_{k}) - \varphi _{k+1}(y_{k+1})\leq \dfrac{(1-d)}{2}\left( \dfrac{1}{\b_k} - \dfrac{1}{\b_{k+1}}\right)\norm {y_{k+1}}^2 ,$
\item[ii)] $\norm {y_{k+1}-y_{k}}^2 \leq \dfrac{\b_{k+1}-\b_k}{\b_{k+1}} \langle  y_{k+1} , y_{k+1}-y_{k} \rangle$,  and 
\item[iii)]  
$ \norm {y_{k+1}-y_{k}} \leq  \dfrac{\b_{k+1}-\b_k}{\b_{k+1}} \norm {y_{k+1}} .$
\end{itemize}
\end{lemma}
\begin{proof}
{\em i)} Recall that $y_{k}$ is the unique minimizer of the  function
$ \varphi _k=f+ \frac{1-d}{2\b_k}\norm{\cdot} ^2 $, then 
\begin{equation*}
\varphi _k (y_{k}) \leq \varphi _k (y_{k+1}) =f(y_{k+1})+\frac{1-d}{2\b_k}\norm{y_{k+1}}^2 ,  
\end{equation*}
and also 
\begin{equation*}
-\varphi _{k+1} (y_{k+1})=-f(y_{k+1})-\frac{1-d}{2\b_{k+1}}\norm{y_{k+1}}^2.
\end{equation*}
Summing these two inequalities, we get the first statement of this lemma.
\\
{\em ii)} By \eqref{nn}, we have
\begin{equation*}
\frac{d-1}{\b_k} y_{k}\in \p f(y_{k})\;\;\;\; \hbox{and}\;\;\;\; \frac{d-1}{\b_{k+1}} y_{k+1}\in\p  f(y_{k+1}).
\end{equation*}
The monotonicity of $\p f$ gives
\begin{equation*}
\left\langle \frac{1-d}{\b_k} y_{k}- \frac{1-d}{\b_{k+1}} y_{k+1} , y_{k+1}-y_{k} \right\rangle \geq 0 ,
\end{equation*}
it follows that
\begin{equation*}
\frac{1-d}{\b_k}\norm{y_{k+1}-y_k}^2\leq \left( \dfrac{1-d}{\b_k} - \dfrac{1-d}{\b_{k+1}}\right) \langle y_{k+1} , y_{k+1}-y_{k} \rangle  ,
\end{equation*}
which gives the second statement of this lemma. The last statement follows from Cauchy-Schwarz inequality.
\end{proof}
By adding the following hypothesis on $(\b_k)$:
\begin{equation*}\label{H02}
\tag*{$(\mathbf{H}_{\b_k})$}\left \{
\begin{array}{ll}
(i) & \dot{\b}_k := \b_{k+1}-\b_k  \neq 0, \hbox{ for } k \hbox{ large enough},\\
(ii)  &\underset{k\rightarrow +\infty}{\lim}\, \dfrac{\dot{\b} _{k+1}}{\dot{\b}_k} =\underset{k\rightarrow +\infty}{\lim}\, \dfrac{\b _{k+1}}{\b_k} =\ell >0;
\end{array}
\right.
\end{equation*}
let us show the main fast convergence results  of this section.
\begin{theorem}\label{Th 2 disc}
Let   $f: \mathcal{H} \rightarrow \mathbb{R}\cup\{ +\infty\}$ be a  proper lower semicontinuous convex function, with $\hbox{argmin}_{\mathcal{H}}f\neq\emptyset$, and $(x_k)$ be a sequence generated by the algorithm \ref{proximal-algorithm}. Let $\r\in]0,1-d[$ and suppose the condition \ref{H02}. 
Then, we have
\begin{itemize}
\item[\textbf{i)}]  for $k$ large enough  
\begin{eqnarray*}
&&f(x_k)-\underset{\mathcal{H}}{\min}f = \mathcal{O} \left( \dfrac{1}{ \beta _k} \right), \;\norm{\dot x_k }^2   = \mathcal{O} \left(    \dfrac{\dot{\b}_k}{\b_k}+e^{-\r k} \right) \\
&& \norm{ x_k -  y_{k} }^2   = \mathcal{O} \left(    \dfrac{\dot{\b}_k}{\b_k}+e^{-\r k} \right) \hbox{ and } d(0,\partial f(x_{k+1})) = \mathcal{O} \left(    \dfrac{1}{\b_k}\right).
\end{eqnarray*}
If $f$ is differentiable on $\mathcal{H}$,   we conclude strong convergence of the gradients to zero with the rate
\begin{equation}\label{eq:grad35}
\norm{ \nabla f(x_k)}  =  \mathcal{O} \left(    \dfrac{1}{\b_k}\right).
\end{equation}

\item[\textbf{ii)}]   Suppose $\ell=1$ in \ref{H02}, then the sequence $(x_k)$ generated by the algorithm \ref{proximal-algorithm} converges strongly to $x^*$  the element of minimum norm of $\hbox{argmin}_{\mathcal{H}}f$ and 
\[
f(x_k)-\underset{\mathcal{H}}{\min} \,f= o \left( \dfrac{1}{\b_k} \right).
\]
\end{itemize}
\end{theorem}
\begin{proof}
To simplify the writing of the proof, we use the following notations: $v_k:=x_k -y_{k-1}$ and
for any sequence $(u_k)$ in $\mathcal{H}$, we write 
$ \dot{u}_k:=u_{k+1}-u_k.$
\\
From \eqref{E_k}, we have
\begin{equation*}
\begin{array}{lll}
\dot{E}_{k} &\leq& \b _{k+1 } \left( \varphi_{k+1 } (x_{k+1 })- \varphi_{k+1 } (y_{k+1 })  \right)- \b _k \left( \varphi_k (x_k)- \varphi_k (y_{k})  \right)\\\\
&\, &+ \frac{\l}{2}\left( \norm {v_{k+1}}^2-\norm {v_k}^2\right)\\\\
&=& \b_k \left[ \varphi_{k} (x_{k+1 })- \varphi_{k } (x_{k })  \right]+\dot{\b}_k \left[ \varphi_{k} (x_{k+1 })- \varphi_{k } (y_k)  \right]\\\\
&\;&+\b_{k+1} \left[ \varphi_{k+1} (x_{k+1})-\varphi_{k} (x_{k+1})+\varphi_k (y_{k}) -\varphi_{k+1} (y_{k+1}  ) \right]\\\\
&\,&+ \frac{\l}{2}\left( \norm {v_{k+1}}^2-\norm {v_k}^2\right).
\end{array}
\end{equation*}
Using the definition of $\varphi_k$ and that $(\b_k)$ is nondecreasing, we deduce that
\begin{equation}\label{5}
\varphi_{k+1} (x_{k+1 })- \varphi_{k } (x_{k+1 })=\dfrac{1-d}{2}\left( \dfrac{1}{\b_{k+1}}-\dfrac{1}{\b_{k}} \right)\norm{x_{k+1}}^2\leq 0 ; 
\end{equation}
so, according to Lemma \ref{lemma 2 discr} $i)$, we obtain
\begin{equation*}
\begin{array}{lll}
\dot{E}_{k} &\leq& \b_k \left[ \varphi_{k} (x_{k+1 })- \varphi_{k } (x_{k })  \right]+\dot{\b} _k\left[ \varphi_{k} (x_{k+1})-\varphi_{k} (y_k)   \right] \\\\
&\;& +\dfrac{(1-d)}{2}\dfrac{\dot{\b}_k}{\b_k} \norm{y_{k+1}}^2 +\dfrac{\l}{2}\left( \norm {v_{k+1}}^2-\norm {v_k}^2\right).
\end{array}
\end{equation*}
Noting that $ \frac{\l}{2}\left( \norm {v_{k+1 }}^2-\norm {v_k}^2\right) \leq \l \langle \dot{v}_k , v_{k+1} \rangle$ then
\begin{equation}\label{6d}
\begin{array}{lll}
\dot{E}_{k} &\leq& \b_k \left[ \varphi_{k} (x_{k+1 })- \varphi_{k } (x_{k })  \right]+\dot{\b}_{k} \left[ \varphi_{k} (x_{k+1})-\varphi_{k} (y_k)   \right] \\\\
&\;& +\dfrac{(1-d)\dot{\b}_k}{2\b_k} \norm{y_{k+1}}^2 +\l \langle \dot{v}_k , v_{k+1} \rangle.
\end{array}
\end{equation}
Returning to form \eqref{disc syst } of algorithm \ref{proximal-algorithm},  we obtain
\begin{equation*}
\begin{array}{lll}
\dot{v}_{k}&=&v_{k+1}-v_{k} =  x_{k+1}-y_{k}- x_k+y_{k-1} = \dot{x}_k - \dot{y}_{k-1}.
\end{array}
\end{equation*}
Thus
\begin{equation}\label{7d}
\l \langle \dot{v}_{k} , v_{k+1}\rangle =\l  \langle\dot x_{k}, x_{k+1}-y_{k} \rangle-\l \langle\dot{y}_{k-1} , x_{k+1}-y_{k} \rangle .
\end{equation}
By strong convexity of $\varphi _k$ and $-\frac{d}{\b_k} \dot{x}_k  \in  \p \varphi_k (x_{k+1})$, we have
\begin{equation}\label{8d}
\left\langle\frac{d}{\b_k} \dot{x}_k, x_{k+1}-y_{k} \right\rangle \leq -\left(\varphi_k(x_{k+1})-\varphi _k (y_{k})\right)-\frac{1-d}{2\b_k}\norm{x_{k+1}-y_{k} }^2.
\end{equation}
Also, using Lemma \ref{lemma 2 discr} $ii)$  and inequality in  \eqref{x^*}, we have for any $b>0$
\begin{equation}\label{9d}
\begin{array}{lll}
-\l   \langle\dot{y}_{k-1} , x_{k+1}-y_{k}\rangle &\leq& \dfrac{\l b}{2}\norm{\dot{y}_{k-1}}^2+\dfrac{\l}{2b}\norm{x_{k+1}-y_{k}}^2\\\\
&\leq& \dfrac{\l b}{2}\dfrac{\dot{\b}^2_{k-1}}{\b^2_{k}}\norm{x^*}^2+\dfrac{\l}{2b}\norm{x_{k+1}-y_{k}}^2.
\end{array}
\end{equation}
Combining \eqref{7d}, \eqref{8d} and \eqref{9d}, we deduce 
\begin{equation}\label{10d}
\begin{array}{lll}
\l \langle \dot{v}_{k} , v_{k+1}\rangle &\leq&\dfrac{-\l \b_k}{d}\left(\varphi_k(x_{k+1})-\varphi _k (y_{k})\right) +\dfrac{\l b}{2}\dfrac{\dot{\b}^2_{k-1}}{\b^2_{k}}\norm{x^*}^2\\\\
& &+\dfrac{\l}{2}\left(\dfrac{1}{b}-\dfrac{1-d}{d} \right)\norm{x_{k+1}-y_{k}}^2.
\end{array}
\end{equation}
Using again the strong convexity of  $\varphi _k$, we have
\begin{equation*}
\varphi _k (x_{k})-\varphi_k(x_{k+1})\geq \left\langle\frac{d}{\b_k} \dot{x}_k, x_{k+1}-x_{k} \right\rangle+ \dfrac{1-d}{2\b_k}\norm{x_{k+1}-x_{k} }^2 =  \dfrac{1+d}{2\b_k}\norm{\dot{x}_k}^2,
\end{equation*}
which implies
\begin{equation}\label{11d}
\begin{array}{lll}
\b_k\left(\varphi_k(x_{k+1})-\varphi _k (x_{k})\right)&\leq& -\dfrac{1+d}{2}\norm{\dot{x}_k}^2 \leq 0.
\end{array}
\end{equation}
Returning to  \eqref{6d},\eqref{10d},\eqref{11d}  and using $\norm{y_{k+1}}\leq\norm{x^*}$, we obtain
\begin{equation}\label{12d}
\begin{array}{lll}
\dot{E}_{k} &\leq&  \frac{\l}{2}\left(\frac{1}{b}-\frac{1-d}{d} \right)\norm{x_{k+1}-y_{k}}^2  +\frac{1}{2}\left( \l b \frac{\dot{\b}^2_{k-1}}{\b^2_{k}}+(1-d)\frac{\dot{\b}_{k}}{\b_k} \right) \norm{x^*}^2 \\\\
&\,&+\left(\b_{k+1}-(1+\frac{\l}{d})\b _k\right)\left[ \varphi_{k} (x_{k+1})-\varphi_{k} (y_k)   \right].
\end{array}
\end{equation}
Set $\m:=\frac{\r}{1-\r}$, then $\m \in \left]0,\frac{1-d}{d}\right[$, and according to the definition of $E_{k}$, we get
\begin{equation*}
\begin{array}{lll}
\m E_{k+1} &=&\m  \b_{k+1}\left( \varphi_{k+1}(x_{k+1})-\varphi_{k+1}(y_{k+1}) \right)+\dfrac{\m \l}{2}\norm{ x_{k+1}-y_{k}}^2 \\\\
&=&\m  \b_{k+1}\left( \varphi_{k}(x_{k+1})-\varphi_{k}(y_k) \right)+\dfrac{\m \l}{2}\norm{ x_{k+1}-y_{k}}^2\\\\
& & +\m  \b_{k+1}\left( \varphi_{k+1}(x_{k+1})-\varphi_{k}(x_{k+1})+\varphi_{k}(y_k)-\varphi_{k+1}(y_{k+1}) \right)
\end{array}
\end{equation*}
Using \eqref{5} and Lemma \ref{lemma 2 discr} $i)$, we obtain
\begin{equation}
\begin{array}{lll}\label{13d}
\m E_{k+1} &\leq&\m  \b_{k+1}\left( \varphi_{k}(x_{k+1})-\varphi_{k}(y_k) \right)+\dfrac{\m (1-d)\dot{\b}_k}{2\b_k}  \norm{x^*}^2 \\\\
&\,&+\dfrac{\m \l}{2}\norm{x_{k+1}-y_{k}}^2.
\end{array}
\end{equation}
Adding  \eqref{12d} and \eqref{13d}, we conclude
\begin{eqnarray}
\dot{E}_{k}+ \m E_{k+1}&\leq&  \dfrac{\l}{2}\underbrace{\left(\dfrac{1}{b}-\dfrac{1-d}{d} +\m\right)}_{A} \norm{x_{k+1}-y_{k}}^2 		\nonumber\\
&\,& +\dfrac{1}{2}\left( \l b \dfrac{\dot{\b}^2_{k-1}}{\b^2_{k}}+(1-d)(1+\m)\dfrac{\dot{\b}_{k}}{\b_k} \right) \norm{x^*}^2 		\nonumber\\
&\,&+ \underbrace{\left((1+\m)\b_{k+1}-\left(1+\frac{\l}{d}\right)\b _k\right)}_{B}\left( \varphi_{k} (x_{k+1})-\varphi_{k} (y_k)   \right).		\label{eq33}
\end{eqnarray}
Choosing $\l$ such that for all $k$: ${\l>d\left( (1+\m)\dfrac{\b _{k+1}}{\b_k}-1 \right)}$, which  is valid since {$\left(\frac{\b _{k+1}}{\b_k}\right)$ is bounded}, we get 
$$
B=\dfrac {\b_k}{d}\left(d\left( (1+\m)\dfrac{\b _{k+1}}{\b_k}-1 \right)-\l\right)\leq 0.
$$ 
Choosing also $b=\dfrac{d}{1-(1+\m)d}$, which is positive  since $\m<\frac{1-d}{d}$,  we cancel the term $A$ in inequality \eqref{eq33}. Consequently, for all $k\geq k_0$  
\begin{equation}\label{14d}
\dot{E}_{k} +\m E_{k+1}\leq \g _k \norm{x^*}^2
\end{equation}
with
\begin{equation*}
\g_k :=\dfrac{1}{2}\left( \dfrac{\l d}{1-(1+\m )d}  \dfrac{\dot{\b}^2_{k-1}}{\b^2_{k}}+(1-d)(1+\m)\dfrac{\dot{\b}_{k}}{\b_k} \right).
\end{equation*}
Inequality \eqref{14d} is equivalent to 
$$
E_{k+1}+\left(\dfrac{\m}{1+\m }-1 \right) E_{k}\leq \left(1-\dfrac{\m}{1+\m } \right) \g_k \norm{x^*}^2. 
$$
Multiplying the last inequality by $e^{\r (k+1)}$ and using $\r = \dfrac{\m}{1+\m}$,  we obtain
\begin{equation*}
e^{\r (k+1)}E_{k+1}+(\r-1 ) e^{\r (k+1)} E_{k}\leq (1-\r)e^{\r (k+1)}\g_{k} \norm{x^*}^2.
\end{equation*}
Hence, for all $k\geq k_0$
\begin{equation*}
e^{\r (k+1)}E_{k+1}-e^{\r k}E_{k} +(e^{-\r}+\r-1)e^{\r (k+1)} E_{k}\leq (1-\r)e^{\r (k+1)}\g_{k} \norm{x^*}^2.
\end{equation*}
Remarking the function $y\mapsto e^{-y}+y-1$ is nonnegative on $\R_+$, we conclude that, for all $k\geq k_0$,
\begin{equation}\label{15d}
e^{\r (k+1)}E_{k+1}-e^{\r k}E_{k} \leq me^{\r (k+1)}\dfrac{\dot{\b}_{k}}{\b_k} \left( r\left(\dfrac{\dot{\b}_{k-1}}{\dot{\b}_k}\right)^2  \dfrac{\dot{\b}_{k}}{\b_k} + 1\right),
\end{equation}
with $r:= \dfrac{\l d }{(1-\r-d)(1-d)}$ and  $m:=\frac12(1-\r)(1-d)(1+\m)\norm{x^*}^2$.\\
Using  the assumption  \ref{H02}, we justify 
$$
\underset{k\to +\infty}{\lim}  \frac{\dot{\b}_k}{\b_k}=\underset{k\to +\infty}{\lim}\left( \frac{\b_{k+1}}{\b_k} -1\right)=\ell -1  \hbox{  and }\underset{k\to +\infty}{\lim} \frac{\dot{\b}_k}{\dot{\b}_{k+1}}  \frac{\b_{k+1}}{\b_{k}}=1 .
$$
Thus,   we have
\begin{eqnarray*}
\underset{k\to +\infty}{\lim} 
\dfrac{e^{\r (k+1)}\dfrac{\dot{\b}_{k}}{\b_k} \left( r\left(\dfrac{\dot{\b}_{k-1}}{\dot{\b}_k}\right)^2  \dfrac{\dot{\b}_{k}}{\b_k} + 1\right)}{\dfrac{\dot{\b}_{k+1}}{\b_{k+1}} e^{\r(k+2)}-\dfrac{\dot{\b}_k }{ \b_{k}}e^{\r (k+1)}} &=&
\underset{k\to +\infty}{\lim}
\dfrac{  r\left(\dfrac{\dot{\b}_{k-1}}{\dot{\b}_k}\right)^2  \dfrac{\dot{\b}_{k}}{\b_k} + 1}{\dfrac{\b_k}{\b_{k+1}}\dfrac{\dot{\b}_{k}}{\dot{\b}_{k+1}} e^{\r}-1} \\
	&=& \dfrac{  \dfrac{r(\ell -1)}{\ell^2}+1}{e^{\r}-1} < +\infty.
\end{eqnarray*}
Therefore, there exist $M>0$ and $k_1\geq k_0$ such that for $k\geq k_1$ 
\begin{equation}\label{16d}
me^{\r (k+1)} \dfrac{\dot{\b}_{k}}{\b_k} \left( r\left(\frac{\dot{\b}_{k-1}}{\dot{\b}_k}\right)^2  \dfrac{\dot{\b}_{k}}{\b_k} + 1\right) \leq  M\left(\dfrac{\dot{\b}_{k+1}}{\b_{k+1}} e^{\r(k+2)}-\dfrac{\dot{\b}_k }{ \b_{k}}e^{\r (k+1)} \right).
\end{equation}
We conclude, for all $k\geq k_1$, 
\begin{equation}\label{15d2}
e^{\r (k+1)}E_{k+1}-e^{\r k}E_{k} \leq M\left(\dfrac{\dot{\b}_{k+1}}{\b_{k+1}} e^{\r(k+2)}-\dfrac{\dot{\b}_k }{ \b_{k}}e^{\r (k+1)} \right) .
\end{equation}
After summing the inequalities \eqref{15d2} between $k_1$ and $k > k_1$, and dividing by $e^{\r (k+1)}$ we get
\begin{equation*}
\begin{array}{lll}
E_{k+1} &\leq& \dfrac{e^{\r k_1}E_{k_1}}{e^{\r(k+1)}}+\dfrac{M}{e^{\r (k+1)}}  \sum\limits_{j=k_1}^k \left( \dfrac{\dot{\b}_{j+1}}{\b_{j+1}}e^{\r (j+2)}-\dfrac{\dot{\b}_j}{\b_j}e^{\r (j+1)}\right)\\
&=& \dfrac{e^{\r k_1}E_{k_1}}{e^{\r(k+1)}}+\dfrac{M}{e^{\r (k+1)}} \left( \dfrac{\dot{\b}_{k+1}}{\b_{k+1}}e^{\r (k+2)}-\dfrac{\dot{\b}_{k_1}}{\b_{k_1}}e^{\r (k_1 +1)}\right)\\
& \leq&\dfrac{e^{\r k_1}E_{k_1}}{e^{\r(k+1)}} + Me^{\r }\dfrac{\dot{\b}_{k+1}}{\b_{k+1}}.
\end{array}
\end{equation*}
Consequently, for $k$ large enough
\begin{equation}\label{estim E_k}
E_{k}=  \mathcal{O}\left( \dfrac{\dot{\b}_{k}}{\b_{k}} + e^{-\r k}  \right).
\end{equation} 
\textbf{i)} Return  to \eqref{con- val disr} and using boundedness of the sequence $ \left( \frac{\dot{\b}_k}{\b_k} \right) _k$
we have $(E_k)$ is bounded, and thus
\begin{equation*}
f(x_k)-\underset{\mathcal{H}}{\min} \,f=\mathcal{O}\left( \dfrac{1}{\b_k} \right).
\end{equation*}
Combining \eqref{conv-tra discr} and \eqref{11d} with \eqref{estim E_k} we get  for $k\geq k_0$ 
$$
\norm{ x_k -  y_{k} }^2 \leq \frac2{1-d}E_k
$$
and
\begin{eqnarray*}
\norm{\dot{x}_k}^2 &\leq& \frac2{1+d} \b_k\left(\varphi_k(x_{k+1})-\varphi _k (x_{k})\right) \\
	&=& \frac2{1+d} \b_k\left(\varphi_k(x_{k+1})-\varphi _k (y_{k})\right)  + \frac2{1+d} \b_k\left(\varphi_k(y_{k})-\varphi _k (x_{k})\right) \\
&\leq& \frac2{1+d} \b_k\left(\varphi_k(x_{k+1})-\varphi _k (y_{k})\right)   \leq \frac2{1+d}E_k.
\end{eqnarray*}
Thus, for $k$ large enough
\begin{equation}\label{32d}
\norm{\dot x_k }^2   = \mathcal{O} \left(    \dfrac{\dot{\b}_k}{\b_k}+e^{-\r k} \right)  \hbox{ and }\norm{ x_k -  y_{k} }^2   = \mathcal{O} \left(    \dfrac{\dot{\b}_k}{\b_k}+e^{-\r k} \right) .
\end{equation}
Now, according to \eqref{disc syst }, let $z_k\in \partial f(x_{k+1})$ satisfying $x_{k+1}+\b_kz_k-dx_{k}=0$. Then
\begin{eqnarray*}
d(0,\partial f(x_{k+1}))^2 &\leq& \|z_k\|^2 = \left\|\frac{d}{\b_k}\dot x_k + \frac{1-d}{\b_k}x_{k+1}\right\|^2 \leq \frac{2d^2}{\b_k^2}\norm{\dot x_k}^2 + \frac{2(1-d)^2}{\b_k^2}\norm{ x_{k+1}}^2.
\end{eqnarray*}
Hence, \eqref{32d} and boundedness of $(x_k)$ and $(\dot x_k)$  lead to 
$$
d(0,\partial f(x_{k+1}))= \mathcal{O} \left(    \dfrac{1}{\b_k}\right), \hbox{  for $k$ large enough}.
$$
If in addition $f$ is differentiable,   we obtain the large convergence rate \eqref{eq:grad35}.\\

\textbf{ii)}   Suppose $\ell=1$ in \ref{H02}, then we have 
$$
\lim_{k\rightarrow +\infty}\dfrac{\dot{\b}_k}{\b_k}= \lim_{k\rightarrow +\infty}\dfrac{{\b}_{k+1}}{\b_k}-1=\ell-1=0,
$$
then $x_k -  y_{k} $ strongly converges to the origine of $\mathcal H$. Return to \eqref{x^*} and \eqref{32d}, we conclude that $x_k$ strongly converges to $x^*$.

Likewise, the use of  \eqref{con- val disr} and   \eqref{estim E_k}    ensures that 
\[
f(x_k)-\underset{\mathcal{H}}{\min} \,f= o \left( \dfrac{1}{\b_k} \right).
\]
\end{proof}
\begin{remark}
Let us first notice that when $\lim_{k\rightarrow +\infty}\dfrac{{\b}_{k+1}}{\b_k}=\ell\neq 1$, then $\lim_{k\rightarrow +\infty}\dfrac{\dot{\b}_{k+1}}{\dot\b_k}= \ell$.

 Indeed, 
\begin{eqnarray*}
\lim_{k\rightarrow +\infty}\dfrac{\dot{\b}_{k+1}}{\dot\b_k} &=& \lim_{k\rightarrow +\infty}\dfrac{{\b}_{k+2}-{\b}_{k+1}}{{\b}_{k+1}-\b_k} =  \lim_{k\rightarrow +\infty}\dfrac{\b_{k+1}}{\b_k}\dfrac{\frac{\b_{k+2}}{{\b}_{k+1}}-1}{\frac{\b_{k+1}}{\b_k}-1} \\
	&=& \ell\dfrac{\ell -1}{\ell -1}=\ell .
\end{eqnarray*}
This proposal is not always guaranteed when $\ell \neq 1$. For example, let $(\b_k)$ be the recurring sequence defined by $\b_0>0$ and $\b_{k+1}=(1+\ell_0^k)\b_k$ for $k\geq 0$ and $0<\ell_0<1$. Then 
\[
\lim_{k\rightarrow +\infty}\dfrac{{\b}_{k+1}}{\b_k}  = \lim_{k\rightarrow +\infty}(1+\ell_0^k)=1=\ell,
\]
nevertheless,
\[
\lim_{k\rightarrow +\infty}\dfrac{\dot{\b}_{k+1}}{\dot\b_k}  = \lim_{k\rightarrow +\infty}(1+\ell_0^k)\ell_0=\ell_0< \ell.
\]
\end{remark}

\section{Application to special cases}\label{sec6}

In this section we review two special cases for the sequence $(\b_k )$.
\subsection{\underline {Case $\b_k = k^m \ln ^q (k) $}:}
 To  investigate the fulfillment of the condition  \ref{H02}, let us first note that for $a>0$ and for $k$ large enough we have 
$$
\;\left(1+\frac{a}{k}\right)^m =1+\frac{am}{k}+o\left( \frac{1}{k} \right)\;\; \hbox{ and } \;\;\ln \left(1+\frac{a}{k}\right) =\frac{a}{k}+o\left( \frac{1}{k} \right). 
$$
Then
\begin{eqnarray*}
\left(1+\frac{a}{k}\right)^m \left(1+\frac{\ln (1+\frac{a}{k})}{\ln (k)}\right) ^q &=&\left(1+\frac{am}{k}+o\left( \frac{1}{k} \right)\right) \left(1+\frac{a}{k\ln (k)}+o\left( \frac{1}{k\ln (k)} \right)\right) ^q\\\\
&=& \left(1+\frac{am}{k}+o\left( \frac{1}{k} \right)\right) \left(1+\frac{aq}{k\ln (k)}+o\left( \frac{1}{k\ln (k)} \right)\right);
\end{eqnarray*}
which gives 
\begin{equation}\label{Tay}
\left(1+\frac{a}{k}\right)^m \left(1+\frac{\ln (1+\frac{a}{k})}{\ln (k)}\right) ^q= 1+\frac{am}{k}+\frac{aq}{k\ln (k)}+o\left( \frac{1}{k\ln (k)} \right).
\end{equation}
\begin{itemize}
\item[$\bullet$]  For \ref{H02}(i), we have $\dot{\b}_ k\neq 0\textit{ for } k>1 $, which  is trivial.\\
\item[$\bullet$]  For \ref{H02}(ii), we come back to \eqref{Tay} to deduce that when $k$ is large enough
\begin{eqnarray*}
 \dfrac{\b_{k+1}}{\b_k}
&=& \dfrac{(k+1)^m \ln ^q  (k+1)}{k^m \ln ^q  (k)}\\
&=&\left( 1+\dfrac{1}{k} \right)^m \left(1+\dfrac{\ln (1+\frac{1}{k})}{\ln (k)} \right)^q\\
&=&1+\frac{m}{k}+\frac{q}{k\ln (k)}+o\left( \frac{1}{k\ln (k)} \right).
\end{eqnarray*}
Thus 
$\lim_{k\rightarrow +\infty}\dfrac{\b_{k+1}}{\b_k}=1>0$.
\\
We also  remark that
\begin{equation*}
\begin{array}{lll}
\dfrac{\dot{\b }_{k+1}}{\dot{\b}_k} &=&\dfrac{(k+2)^m\ln^q (k+2)-(k+1)^m\ln^q (k+1)}{(k+1)^m\ln^q (k+1)-k^m\ln^q(k)}\\\\
&=& \dfrac{\left(1+\frac{2}{k}\right)^m \left(1+\frac{\ln (1+\frac{2}{k})}{\ln (k)}\right) ^q-\left(1+\frac{1}{k}\right)^m\left(1+\frac{\ln (1+\frac{1}{k})}{\ln (k)}\right) ^q}{\left(1+\frac{1}{k}\right)^m\left(1+\frac{\ln (1+\frac{1}{k})}{\ln (k)}\right) ^q-1}.
\end{array}
\end{equation*}  
Then, using \eqref{Tay}, we get for $k$ large enough
\begin{equation*}
\begin{array}{lll}
\dfrac{\dot{\b }_{k+1}}{\dot{\b}_k} &=& \dfrac{\!\left(\!1+\frac{2m}{k}+\frac{2q}{k\ln (k)}+o\left(\! \frac{1}{k\ln (k)} \!\right) \right)-\left(\! 1+\frac{m}{k}+\frac{q}{k\ln (k)}+o\left( \frac{1}{k\ln (k)} \!\right)\!\right) }{\left(\!1+\frac{m}{k}+\frac{q}{k\ln (k)}+o\!\left(\! \frac{1}{k\ln (k)} \!\right)\!\right) -1}\\
&=& \dfrac{m\ln (k) +q +o(1) }{m\ln (k) +q +o(1)}.
\end{array}
\end{equation*} 
Thus, for $(m,q)\in (\mathbb{R}^+)^2 \setminus \{(0,0)\}$, we deduce $\underset{k \to +\infty}{\lim}\dfrac{\dot{\b }_{k+1}}{\dot{\b}_k}=1. $
\end{itemize}
Consequently all conditions of \ref{H02} are satisfied. 
So, based on Theorem \ref{Th 2 disc}, the following result can be proved  easily.
\begin{proposition}\label{corollary 1 ln}
Let $f$  and $(x_k)$ be as in Theorem \ref{Th 2 disc}, and $\b_k = k^m \ln ^q  (k) $ where $(m,q)\in (\mathbb{R}^+)^2 \setminus \{(0,0)\}$. Then $(x_k)$  converges strongly to $x^*$  the element of minimum norm of $\hbox{argmin}_{\mathcal{H}}f$ and
 \begin{eqnarray*}
&& f(x_k)-\underset{\mathcal{H}}{\min}f = o \left( \frac{1}{k^m \ln^q(k)} \right);\;\;\;\;
\norm{ \dot x_k}^2   =\left \{
\begin{array}{lcl}
 \mathcal{O} \left( \frac{1}{k} \right) \;\;&\textit{ if }\; r\neq 0&  \\
 \mathcal{O} \left( \frac{1}{k\ln (k)} \right)\;\; &\textit{ if }\; r= 0;& \\
\end{array}
\right. \\ 
&&d(0,\partial f(x_{k}))= \mathcal{O} \left( \frac{1}{k^m \ln^q(k)} \right).
 \end{eqnarray*}
If moreover  $f$ is differentiable, then\; 
\begin{eqnarray*}
&&\; \norm{ \nabla f(x_k)}  =  \mathcal{O} \left( \dfrac{1}{k^{m}\ln ^{q}(k)} \right).
 \end{eqnarray*}
\end{proposition}
\subsection{\underline {Case $ \beta_k = k^me^{\gamma k^r}$}:\;}
  Let us  now  treat   the case $ \beta_k = k^me^{\gamma k^r}$ with $r\in ]0,1], m\geq 0$ and $\gamma >0$.  To  fulfill   the condition  \ref{H02}, we first remark that for $0<r<1$ and $a>0$ we have  for $k$ large enough
\begin{equation*}\label{Tay 2}
(k+a)^r -k^r = k^r \left[ \left( 1+\dfrac{a}{k} \right)^r -1  \right] =ark^{r-1}+o(k^{r-1}).
\end{equation*} 
Hence,   for $k$ large enough
\begin{equation}\label{Tay 3}
e^{\g((k+a)^r-k^r)}=e^{a\g r k^{r-1}+o(k^{r-1})}=1+a\g rk^{r-1}+o(k^{r-1}).
\end{equation}
\begin{itemize}
\item[$\bullet\,$] Likewise \ref{H02}(i), $\dot{\b}_ k\neq 0\textit{ for } k>1 $ is trivial.
\item[$\bullet\,$] For \ref{H02}(ii), we distinguish two cases:
\\
$\star$ If $0<r<1$  then using \eqref{Tay 3} with $a=1$, we get for $k$ large enough 
 \begin{equation*}
 \dfrac{\b_{k+1}}{\b_k}= \left(1+\frac1{k}\right)^me^{\gamma \left((k+1)^r-k^r\right)}= 1+\g rk^{r-1}+o(k^{r-1}),
\end{equation*}  
\begin{equation*}
\begin{array}{lll}
\dfrac{\dot{\b }_{k+1}}{\dot{\b}_k}&=& \dfrac{(k+2)^me^{\g(k+2)^r}-(k+1)^me^{\g(k+1)^r}}{(k+1)^me^{\g(k+1)^r}-k^me^{\g k^r}}\\\\
&=&\dfrac{1+\frac{m}{\g rk^r}+o\left(\frac{m}{k^{-r}}\right)}{1+\frac{m}{\g rk^r}+o\left(\frac{m}{k^{-r}}\right)}.
\end{array}
\end{equation*}
Thus 
$$ 
\underset{k \to +\infty}{\lim}\dfrac{{\b }_{k+1}}{{\b}_k}=\underset{k \to +\infty}{\lim}\dfrac{\dot{\b }_{k+1}}{\dot{\b}_k}=1>0. 
$$
$\star$  If $r=1$, we have for $k$ large enough the termes  $ \dfrac{\b_{k+1}}{\b_k}$  and $\dfrac{\dot{\b }_{k+1}}{\dot{\b}_k}$ are equivalent to $e^{\g}\left( 1+\frac{m}{k}+o\left(\frac{1}{k}\right)\right)$, so 
$$
\underset{k \to +\infty}{\lim}\dfrac{{\b }_{k+1}}{{\b}_k} =\underset{k \to +\infty}{\lim}\dfrac{\dot{\b }_{k+1}}{\dot{\b}_k}=e^{\g}>1. 
$$
\end{itemize} 
Applying again  Theorem \ref{Th 2 disc}, we obtain 
\begin{proposition}\label{corollary 2 disc}
Let $f$  and $(x_k)$ be as in Theorem \ref{Th 2 disc}, where $\b_k =   k^me^{\gamma k^r}.$ 
Then
\\
$\star$  if $r=1$ and $m=0$, we have  $(x_k)$ is bounded and   for $k$ large enough
\begin{eqnarray*}
&&f(x_k)-\underset{\mathcal{H}}{\min}f = \mathcal{O} \left( e^{-\gamma k} \right);
\end{eqnarray*}
$\star$ if either $0<r<1$ or $m>0$,  $(x_k)$  converges strongly to $x^*$  the element of minimum norm of $\hbox{argmin}_{\mathcal{H}}f$, and for $k$ large enough, we have   
\begin{eqnarray*}
&& f(x_k)-\underset{\mathcal{H}}{\min}f = o \left( k^{-m}e^{-\gamma k^r} \right), \;
\norm{\dot x_k}^2   = \mathcal{O} \left( \dfrac{1}{k^{1-r}} \right)\;\\
&&d(0,\partial f(x_{k}))=  \mathcal{O} \left( k^{-m}e^{-\gamma k^r} \right).
 \end{eqnarray*}
If moreover  $f$ is differentiable, then\; 
\begin{eqnarray*}
\hbox{and }\; \norm{\n f(x_{k})}  =    \mathcal{O} \left( k^{-m}e^{-\gamma k^r} \right).
\end{eqnarray*}
\end{proposition}

\section{Numerical example}\label{sec7}

In this section, we present a numerical example in the framework of nondifferentiable minimization problem to illustrate the performance of our iterative method \ref{proximal-algorithm}. So, we consider the proper lower semicontinuous and convex function 
$$
f(x,y)= |x| + \iota_{[-a,a]}(x) + \frac12(y-v_0)^2
$$
where $a>0, v_0\in \mathbb R$ and $\iota_{[-a,a]}$ is the indicator function having the value $0$ on $[-a,a]$ and $+\infty$ elsewhere.   The well-known Fermat's and subdifferential sum rules for convex functions ensure that
$
(\bar x,\bar y) \in\hbox{argmin}_{\mathbb R^2}f ,
$
iff 
$$
(0,0)\in \partial f(\bar x,\bar y) = \partial\left( |\cdot| + \delta_{[-a,a]}\right)(\bar x)\times \partial\left( \frac12(\cdot-v_0)^2\right)(\bar y)  \iff (\bar x,\bar y)=(0,v_0).
$$
Thus $\hbox{argmin}_{\mathbb R^2}f = \{0,v_0)\}$.
Using the rules for calculating proximal operators, see \cite[Chap 6]{beck}, we have  for each $\lambda >0, (x,y)\in \mathbb R^2$,
\begin{eqnarray*}
\hbox{prox}_{\lambda f}(x,y) &=& \left( \hbox{prox}_{\lambda \left( |\cdot| + \delta_{[-a,a]}\right)}(x)\, , \, \hbox{prox}_{\lambda \left( \frac12(\cdot-v_0)^2\right)}(y)\right)\\
	&=&  \left(\min[\max(|x|-\l,0),a]\hbox{sign}(x)\; , \; -\frac1{\l + 1}(y-\l v_0)\right)
\end{eqnarray*}
Below, we explain our algorithm \ref{proximal-algorithm} and the one proposed by L\'aszl\'o in the recent paper \cite{Las23}:
\begin{eqnarray}
\label{eq:BCR}&&\left\{
\begin{array}{rll}
(x_0,y_0)&\in& \mathbb R^2,\\
(x_{k+1},y_{k+1})&=&\hbox{prox}_{\b_k f}(d x_k,dy_k)
\end{array}\right.
\\
\label{eq:L}&&\left\{
\begin{array}{rll}
(x_0,y_0)&,& (x_1,y_1)\in\mathbb R^2,\\
(u_k,v_k)&=& (x_k,y_k)+(1-\frac{\a}{k^q})(x_k-x_{k-1},y_k-y_{k-1}),\\
(x_{k+1},y_{k+1})&=&\hbox{prox}_{\l_k f}\left((u_k,v_k)-\frac{c}{k^p}(x_k,y_k)\right)
\end{array}\right.
\end{eqnarray}
We note that \eqref{eq:L} is an implicit discretization of the differentiable dynamical system studied in \cite{Las23A}, that is $\ddot x(t)+ \frac{\alpha}{t^q} \dot x(t)+\nabla f(x(t))+ \frac{c}{t^p} x(t)=0$ where $\alpha,c, q,p>0$. 
In \cite[Theorem 1.1]{Las23}, for $\l_k=\l k^\delta$, the conditions on the parameters $p,q,\alpha, \delta, \l$ and $c$ impose that $0< q<1, p\leq 2, \l>0, \d\in\mathbb R, c>0$ where the choice of $\d,\l$ and $c$ depend on the positioning of $p$ with respect to $q +1$.
\begin{figure} 
 \includegraphics[scale=0.35]{DiscrOrder1-2_GeneralConvB.pdf}
  \caption{Comparison of convergence  rates of values, trajectories and gradients between teh proposed algorithm  \eqref{eq:BCR} and that of L\'aszl\'o \eqref{eq:L}, see \cite{Las23}.}
 \label{fig:trigs-d2} 
\end{figure}
The squared distance of $(x_k,y_k)$, generated by the above two iterations, to the  solution $(x^*,y^*)$ and the decay of the objective function $f(x_k,y_k)-\min_{\mathbb R^2}f(x,y)$ along the iterations are shown in Figure \ref{fig:trigs-d}  for the selected values $\b_k$ equal respectively to $k^3\ln^3k$ and $k^3e^{2 k^{4/5}}$ for the algorithm \eqref{eq:BCR}, and those equal to $p=2, q=\frac45, \alpha=2, \delta=2, \l=5$ for the algorithm \eqref{eq:L}. 
We note that the choice of $\beta_k=k^3e^{2 k^{4/5}}$ justifies the originality of  Theorem \ref{Th 2 disc}, since  we end up with an exponential convergence rate for the values and the gradient. Also,  the benefit of the inverse-relaxation $\beta_k=k^3e^{2 k^{4/5}}$ (red curve  in Figure \ref{fig:trigs-d2}), as allowed in Proposition \ref{corollary 2 disc}, is clearly visible.
\\
\begin{figure} 
 \includegraphics[scale=0.35]{cont-ordre1_Example1_VarisParam}
  \caption{Results of the convergence rates of values, trajectories and gradients when the parameters $m, r, \g, d$ in the algorithm \ref{proximal-algorithm} vary. Note that all these parameters are active in the convergence evolution except that of $d$.}
 \label{fig:trigs-param} 
\end{figure}
Figure \ref{fig:trigs-param} explains the interest in the growth of the three parameters $m,r,\gamma$ in $\beta_k=k^me^{\gamma k^r}$, while the parameter $d$ in Algorithm \ref{proximal-algorithm} does not act on the convergence rates.
}\fi


\section{Discussion and Conclusion}\label{sec13}

In this paper we first study a first order dynamical system 
\[
\dot{x}(t)  + \beta (t)\n\left(f + \dfrac{c}{2\beta (t)}\norm{\cdot} ^2\right)(x(t))= 0
\]
which can be viewed as a Cauchy system with a time rescaling parameter on the gradient  of the objective function and a Tikhonov regularization term. This dynamical system  makes it possible to maintain the rate for convergence of values for generated trajectory for systems without Tikhonov regularization. We obtain moreover a similar rate for convergence of the norm-square of the associated gradients towards the origin.
Another original new assertion of our system is the strong convergence of the  generated trajectory towards the minimum norm solution of the objective function.


It would be a novelty in the literature to achieve a good speed of convergence for values and gradients with a time first order system, because recent papers \cite{BCR1,BCR2} have achieved the same speeds of convergence with the time second order differential systems:
\[
\ddot{x}(t) +  \alpha \, \dot{x}(t) + \beta (t) \nabla f(x(t)) + c x(t)    =0,
\] 
\[
 \ddot{x}(t) +  \alpha \, \dot{x}(t) + \d \nabla ^2 f(x(t))\dot{x}(t) + \beta (t) \nabla f(x(t)) + c x(t)    =0.
 \] 

As desired, future 
 theoretical and numerical study  may be the subject of further work
  which focus on extending the study of the performance of systems  associated with other penalization functions than $\frac12\|\cdot\|^2$. More precisely, consider 
\[
\dot{x}(t)  + \beta (t)\n\left(f + \dfrac{1}{\beta (t)}g\right)(x(t))= 0
\]
 in order to solve with the same performance the two level hierarchical minimization problem: $\min g(x) $ under constraints the solution-set of the convex function $f$.

\if{
\section{Declarations}\label{sec14}
Some journals require declarations to be submitted in a standardised format. Please check the Instructions for Authors of the journal to which you are submitting to see if you need to complete this section. If yes, your manuscript must contain the following sections under the heading ?Declarations?:

\paragraph{Funding.} 
This research project is supported by the authors themselves

\paragraph{Conflict of interest.}
 Authors declared no competing interest. 
\paragraph{Ethics approval.} 
Not applicable
\paragraph{Consent to participate.} 
Not applicable
\paragraph{Consent for publication.} 
Not applicable
\paragraph{Availability of data and materials.} 
The data sets used are from the authors and can be provided based on request
}\fi


%
\medskip

\begin{appendices}
\section{}\label{sec-appendix}
\if{
We rely on the basic properties of the Moreau envelope $f_{\g}: \mathcal{H} \longrightarrow \mathbb{R}  \;\; (\g >0)$, which is defined by
\begin{equation}\label{1s}
f_\g (x) = {\min}_{y \in \mathcal{H}} \left( f(y)+\frac{1}{2\g }\norm{x-y}^2 \right) \;\; \hbox{for any } x\in \mathcal{H}.
\end{equation}
Recall that,  the functions $f$ and $f_\l$ share the same optimal objective value and the same
set of minimizers
$$ {\min}_{y \in \mathcal{H}}\, f = {\min}_{y \in \mathcal{H}}\, f_\g \;\;\;\;\;\; \hbox{and}\;\;\;\;\;\; {\argmin}_{\mathcal{H}}\, f ={\argmin}_{\mathcal{H}}\, f_\g . $$
In addition, $f_\l$ is convex and continuously differentiable whose gradient is $\frac{1}{\g}$-Lipschitz continuous.
\begin{lemma}[{\rm\cite{Att2}, \cite[Section 12]{BaCo}, \cite[Prop 2.6]{Brezis}}]\label{lem-basic-c}
The unique point where the minimum value is achieved in \eqref{1s} is denoted by ${\prox} _{\g f} (x) $, and satisfies the following classical formulas: For each $\g >0$ and $ x\in \mathcal{H}$, 
\begin{enumerate}
\item $ f_\g (x)=f({\prox} _{\g f} (x)) +\frac{1}{2\g}\norm{x-{\prox} _{\g f} (x)}^2  $;
\item  $\nabla f_\g (x) = \frac{1}{\g} (x-{\prox} _{\g f} (x)) $;
\item  ${\prox} _{\theta f_\g}(x)=\frac{\g}{ \g +\theta} x + \frac{\theta}{ \g +\theta} {\prox} _{(\g + \theta)f}(x), \;\hbox{ for all } \theta > 0$;
\item  $  \|{\prox} _{\g f} (0)\|\leq \|x^{*}\|$, where  $x^{*}=\mbox{\rm proj}_{\argmin f} 0$; \label{2a}  \medskip
\item   $\lim_{\g \rightarrow +\infty}\|{\prox} _{\g f} (0)-x^{*}\|=0 $. \label{2b}
\end{enumerate}
 \end{lemma}
}\fi

The following Lemma provides an extended  version of the classical gradient lemma which is valid for differentiable convex functions. The following version has been obtained in \cite[Lemma 1]{ACFR}, \cite{ACFR-Optimisation}.
We reproduce its proof for the convenience of the reader.
\begin{lemma}\label{ext_descent_lemma}
Let  $f: \cH \to \R$ be  a  convex function whose gradient is $L$-Lipschitz continuous. Let $s \in ]0,1/L]$. Then for all $(x,y) \in \cH^2$, we have
\begin{equation}\label{eq:extdesclem}
f(y - s \nabla f (y)) \leq f (x) + \left\langle  \nabla f (y), y-x \right\rangle -\frac{s}{2} \|  \nabla f (y) \|^2 -\frac{s}{2} \| \nabla f (x)- \nabla f (y) \|^2 .
\end{equation}
In particular, when $\argmin f \neq \emptyset$, we obtain that for any $x\in \cH$
\begin{equation}\label{eq:extdesclemb}
f(x)-\min_{\cH} f  \geq \frac{1}{2L} \| \nabla f (x)\|^2 .
\end{equation}
\end{lemma}
\begin{proof}
Denote $y^+=y - s \nabla f (y)$. By the standard descent lemma applied to $y^+$ and $y$, and since $sL \leq 1$ we have
\begin{equation}\label{eq:descfm2}
f(y^+) \leq f(y) - \frac{s}{2}\left(2-Ls\right) \| \nabla f (y) \|^2 \leq f(y) - \frac{s}{2} \|  \nabla f (y) \|^2.
\end{equation}
We now argue by duality between strong convexity and Lipschitz continuity of the gradient of a convex function. Indeed, using Fenchel identity, we have
\[
f(y) = \dotp{\nabla f(y)}{y} - f^*(\nabla f(y)) .
\]
$L$-Lipschitz continuity of the gradient of $f$ is equivalent to $1/L$-strong convexity of its conjugate $f^*$. This together with the fact that $(\nabla f)^{-1}=\partial f^*$ gives for all $(x,y) \in \cH^2$,
\[
f^*(\nabla f(y)) \geq  f^*(\nabla f(x)) + \dotp{x}{\nabla f(y)-\nabla f(x)} + \frac{1}{2L}\norm{\nabla f(x)-\nabla f(y)}^2 .
\]
Inserting this inequality into the Fenchel identity above yields
\begin{align*}
f(y) 
&\leq - f^*(\nabla f(x)) + \dotp{\nabla f(y)}{y} - \dotp{x}{\nabla f(y)-\nabla f(x)} - \frac{1}{2L}\norm{\nabla f(x)-\nabla f(y)}^2 \\
&= - f^*(\nabla f(x)) + \dotp{x}{\nabla f(x)} + \dotp{\nabla f(y)}{y-x} - \frac{1}{2L}\norm{\nabla f(x)-\nabla f(y)}^2 \\
&= f(x) + \dotp{\nabla f(y)}{y-x} - \frac{1}{2L}\norm{\nabla f(x)-\nabla f(y)}^2  \\
&\leq f(x) + \dotp{\nabla f(y)}{y-x} - \frac{s}{2}\norm{\nabla f(x)-\nabla f(y)}^2 .
\end{align*}
Inserting the last bound into \eqref{eq:descfm2} completes the proof.
\end{proof}

\end{appendices}

\end{document}